\documentclass[reqno,11pt]{amsart}
\usepackage[top=2.0cm,bottom=2.0cm,left=3cm,right=3cm]{geometry}
\usepackage{amsthm,amsmath,amssymb,dsfont}
\usepackage{mathrsfs,amsfonts,functan,extarrows,mathtools}

\usepackage{marginnote}
\usepackage{xcolor}
\usepackage{stmaryrd}
\usepackage{esint}
\usepackage{graphicx}
\usepackage{bm}

\usepackage{color}
\allowdisplaybreaks
\theoremstyle{plain}
\newtheorem{defn}{Definition}[section]
\newtheorem{thm}{Theorem}[section]
\newtheorem{cor}[thm]{Corollary}
\newtheorem{lem}[thm]{Lemma}
\newtheorem{prop}[thm]{Proposition}
\newtheorem{rem}{Remark}[section]

\numberwithin{equation}{section}

\newcommand{\dv}{{\rm div\,}}
\newcommand{\eps}{{\epsilon}}
\newcommand{\dd}{{\ \rm d}}

\begin{document}

\title[Dissipative solutions for the compressible Navier-Stokes equations]
{Dissipative solutions to the compressible isentropic Navier-Stokes equations}

\author[Liang Guo]{Liang Guo}
\address[Liang Guo]{Department of Mathematics, Nanjing University, Nanjing
 210093, P. R. China}
\email{guoliang142857@163.com}

\author[Fucai Li]{Fucai Li}
\address[Fucai Li]{Department of Mathematics, Nanjing University, Nanjing
 210093, P. R. China}
\email{fli@nju.edu.cn}

\author[Cheng Yu]{Cheng Yu}
\address[Cheng Yu]{Department of Mathematics, University of Florida, Gainesville, FL 32611, United States of America}
\email{chengyu@ufl.edu}

\begin{abstract}

The existence of dissipative solutions to the compressible isentropic Navier-Stokes equations  was established in this paper.  This notion was inspired by the concept of dissipative solutions to the incompressible Euler equations of Lions (\cite{Lions-1996}, Section 4.4). Our method is to
recover such solutions by passing to the limits from  approximated solutions, thanks to  compactness argument.

\end{abstract}



\maketitle

\section{Introduction}\label{Sec:intro}

\setcounter{equation}{0}
 \indent \allowdisplaybreaks

This paper aims to study the existence of dissipation solutions to the compressible Navier-Stokes equations, which was inspired by the work of Lions \cite{Lions-1996}.
In particular, Lions introduced the concept of dissipation solutions to the incompressible Euler equations and proved its existence. This can imply the property of weak-strong uniqueness of the Euler equations. In this paper, we are particularly interested in extending these results to the compressible Navier-Stokes equations.

Thus, we consider the following  compressible isentropic Navier-Stokes equations over $\mathbb{R}^{+} \times \Omega$ $( \Omega \subset \mathbb{R}^3 )$:
\begin{align}\label{com-NS}
\left\{
 \begin{array}{ll}
  \partial_{t}\rho + \dv (\rho{\bf u})=0 \,, \\
  \partial_{t}(\rho {\bf u}) + \dv ( \rho {\bf u} \otimes {\bf u} )+\nabla {\rm p} (\rho) = \dv \mathbb{S}(\nabla{\bf u}) \,,
 \end{array}\right.
\end{align}
where $\rho$ denotes the density, ${\bf u} \in \mathbb{R}^3$ the velocity and $ {\rm p} (\rho) = A \rho^\gamma $ the pressure with the constant $A > 0$ and the adiabatic exponent $\gamma > 1$, respectively. The viscous stress tensor $\mathbb{S}$ satisfies the Newton's rheological law:
\begin{equation*}
  \mathbb{S}( \nabla {\bf u} ) = \mu ( \nabla {\bf u} + \nabla^{\top} {\bf u}) + \lambda ( \dv {\bf u} ) {\bf I}_{3} \,,
\end{equation*}
where the constants $\mu$ and $\lambda$ are the Lam\'{e} viscosity coefficients of the flow satisfying $\mu > 0$ and $2 \mu + 3 \lambda \geq 0$, and ${\bf I}_{3}$ is the $3 \times 3$ identity matrix. The equations \eqref{com-NS} are supplement with the initial data
\begin{equation}\label{initial}
  \begin{aligned}
    ( \rho, \rho {\bf u} ) |_{ t = 0 } = ( \rho_0, {\bf m}_0 ) \,,
  \end{aligned}
\end{equation}
and as one of the following boundary conditions:

1. the periodic case
\begin{equation}\label{boundary1}
  \begin{aligned}
    \Omega = \mathbb{T}^3 = \mathbb{R}^3/(-\pi, \pi)^3 \,;
  \end{aligned}
\end{equation}

2. the Dirichlet boundary condition
\begin{equation}\label{boundary2}
  \begin{aligned}
    {\bf u} |_{ \partial \Omega } = 0,
  \end{aligned}
\end{equation}
where $\Omega$ is a bounded domain in $\mathbb{R}^3.$

For the weak solutions of the compressible Navier-Stokes equations,
 Lions \cite{Lions-1998} introduced the concept of renormalized solutions. This allows him to establish
the global existence of weak solutions with large data for any $\gamma\geq \frac{9}{5}$. Later, his result was improved in \cite{FNP-jmfm-2001} by extending the value of the adiabatic exponent  to $\gamma>\frac{3}{2}.$
Improving the range of  $\gamma$
is an interesting and fundamental problem, which certainly is from a physical viewpoint. It is also a challenging problem in mathematics, since the restriction on $\gamma>\frac{3}{2}$
is absolutely essential to the analysis in \cite{Lions-1998, FNP-jmfm-2001}. This current paper aims to build up solutions in a weaker sense than in the renormalized sense,  which was inspired by the concept of  dissipation solutions in Lions  \cite{Lions-1996}.

DiPerna and Majda \cite{DM-cmp-1987} proposed a measure-valued solution, by a generalized Young measure and proved the global existence of such solution to the incompressible Euler equations with any initial data. However, they have not investigated the weak-strong uniqueness structure. The concept of dissipation solutions to the incompressible Euler equations was introduced and its existence was given in  \cite{Lions-1996}. This solution can imply the
 weak-strong uniqueness property. Meanwhile,  Bellout et al. \cite{BCN-siam-2002} also proposed  a \emph{very weak} $L^2$ solution.  Later, Brenier et al. \cite{BDS-2011}  established the weak-strong uniqueness of the admissible measure-valued solutions to the incompressible Euler equations, and showed  that the admissible measure-valued solution is also a dissipative solution in the sense of Lions.

The admissible assumption is to say that the kinetic energy is always less than or equal to the initial energy, which plays a key role in  \cite{BDS-2011}. Under the admissible assumption, De Lellis and Sz\'{e}kelyhidi (\cite{DS-arma-2010}, Proposition 1) proved that the weak solution of incompressible Euler equations is a dissipative solution in the sense of Lions. Gwiazda et al. \cite{GSW-nonlinearity-2015} extended the measure-valued solutions to some compressible fluid models and proved the weak-strong uniqueness of the admissible measure-valued solutions to the isentropic Euler equations in any space dimension. And afterwards, Feireisl et al. \cite{FGSW-cv-2016} introduced a dissipative measure-valued solution to the compressible barotropic Navier-Stokes system, and proved the existence for the adiabatic exponent $\gamma > 1$ and the weak-strong uniqueness property of the dissipative measure-valued solutions.

The main contribution of this paper is to extend the notion of dissipative solutions to the incompressible Euler equations in \cite{Lions-1996} to the equations of the compressible fluids. To compared to the incompressible case, we have to consider the density effect and the mass equation. To this end, we define the following two smooth functions
\begin{equation}\label{E1-2}
  \begin{aligned}
    E_1 ( r, {\bf v} ) = \partial_t r + \dv ( r {\bf v} ) \,, \quad E_2 ( r, {\bf v} ) = r \partial_t {\bf v} + r {\bf v} \cdot \nabla {\bf v} + \nabla {\rm p} (r) - \dv \mathbb{S} ( \nabla {\bf v} ) \,,
  \end{aligned}
\end{equation}
where $r$ (has a positive lower bound $r_1$) and ${\bf v}$ be two smooth functions on $[0, \infty) \times \Omega$. This allows us to derive a priori relative entropy \eqref{relative-energy-ineq}. It is crucial to have  \eqref{relative-energy-ineq} for giving the definition of  dissipative solutions to the compressible Navier-Stokes equations,  which implies the weak-strong uniqueness property.

The main idea is to build up the solutions of  the modified Brenner model \eqref{approxi-system}, and to show that this approximated solution is also a dissipative solution to  \eqref{approxi-system}. Then we can recover our solution by passing to the limits from this dissipative solution, thanks to  the compactness argument.

This paper is organized as follows. In Section \ref{Sec:def-result}, we derive some a priori estimates,  introduce the definition of dissipative solutions to the compressible isentropic Navier-Stokes equations and state our main results: the existence of dissipative solutions (Theorem 2.1) and the weak solution is also a dissipative solution (Theorem 2.3). In Section \ref{Sec:regular}, we show that the smooth functions $r$ and ${\bf v}$ can be replaced by a class of functions with lower regularities. This can be done through  the regularization procedure. In Section \ref{Sec:thm-exist}, we give the proof of Theorem \ref{thm-existence-dissipative-sol} by using compactness analysis on the approximated solutions. Section \ref{Sec:ws-ds} is the proof of Theorem \ref{thm-weak-dissipative}. Note that Corollary \ref{weak-strong-ds} and Corollary \ref{weak-strong-ws} are  direct conclusions of Theorem \ref{thm-existence-dissipative-sol} and Theorem \ref{thm-weak-dissipative}, respectively, thus we omit their proofs. In Appendix \ref{Sec:appendix}, we collect some auxiliary lemmas.

\section{Definition of dissipative solutions and the main results}\label{Sec:def-result}
The main goal of this section is to define our  dissipative solutions to the compressible Navier-Stokes equations and address our main results. To this end, we start with deriving some a priori estimates which are crucial to our definition. Thus, we  assume the solutions are smooth here.

Multiplying the continuity equation $\eqref{com-NS}_1$ by $\frac{1}{2} |{\bf v}|^2 - {\rm P}'(r)$ (${\rm P}'(r) = \frac{A \gamma}{\gamma-1} r^{\gamma - 1}$), integrating the result over $\Omega$, taking the inner product of the momentum equation $\eqref{com-NS}_2$ with ${\bf u} - {\bf v}$, by integration by parts, and adding them up gives
\begin{align}\label{energy}
    & {\rm \frac{d}{dt}} \int \frac{1}{2} \rho | {\bf u} - {\bf v} |^2 + {\rm P}(\rho) - {\rm P}'(r) \rho \dd x + \int \mathbb{S} (\nabla {\bf u} ) : \nabla {\bf u} \dd x \nonumber\\
    =& - \int \rho ( {\bf u} - {\bf v} ) \cdot \nabla {\bf v} \cdot ( {\bf u} - {\bf v} ) \dd x + \int \rho ( \partial_t {\bf v} + {\bf v} \cdot \nabla {\bf v} ) \cdot ( {\bf v} - {\bf u} ) \dd x \nonumber\\
    & - \int \rho \partial_t {\rm P}'(r) \dd x - \int \rho {\bf u} \cdot \nabla {\rm P}'(r) \dd x \nonumber\\
    & - \int {\rm p}(\rho) \dv {\bf v} \dd x + \int \mathbb{S} (\nabla {\bf u}) : \nabla {\bf v} \dd x \,,
\end{align}
where ${\rm P}(s) = \frac{A}{\gamma-1} s^\gamma$. Note that  ${\rm P}'(r) r - {\rm P}(r) = {\rm p}(r) = A r^{\gamma}$ and ${\rm P}''(r) r = {\rm p}'(r)$, this gives
\begin{align}\label{pressure}
    {\rm \frac{d}{dt}} \int {\rm P}'(r) r - {\rm P}(r) \dd x = & \int \partial_t {\rm p}(r) + \dv ( {\rm p}(r) {\bf v} ) \dd x \nonumber\\
    = & \int {\rm p}'(r) \partial_t r + {\bf v} \cdot \nabla {\rm p}(r) + {\rm p}(r) \dv {\bf v} \dd x \nonumber\\
    = & \int r \partial_t {\rm P}'(r) \dd x + \int r {\bf v} \cdot \nabla {\rm P}'(r) + {\rm p}(r) \dv {\bf v} \dd x \,.
\end{align}

Using\eqref{E1-2},   \eqref{energy} and \eqref{pressure}, one obtains that
\begin{align}\label{relative-energy-E1-2}
    & {\rm \frac{d}{dt}} \int \frac{1}{2} \rho | {\bf u} - {\bf v} |^2 + {\rm P}(\rho) - {\rm P}'(r)( \rho - r ) - {\rm P}(r) \dd x + \int \mathbb{S} (\nabla {\bf u} - \nabla {\bf v} ) : \nabla ( {\bf u} - {\bf v} ) \dd x \nonumber\\
    \leq & - \int \rho ( {\bf u} - {\bf v} ) \cdot \mathbb{D} ( {\bf v} ) \cdot ( {\bf u} - {\bf v} ) \dd x - \int ( {\rm p}(\rho) - {\rm p}'(r) (\rho - r) - {\rm p}(r) ) \dv {\bf v} \dd x \nonumber\\
    & + \int ( r - \rho ) {\rm P}''(r) E_1 (r, {\bf v}) \dd x + \int \frac{\rho}{r} E_2 (r, {\bf v}) \cdot ( {\bf v} - {\bf u} ) \dd x \nonumber\\
    & + \int \frac{\rho - r}{r} \dv \mathbb{S} (\nabla {\bf v}) \cdot ( {\bf v} - {\bf u} ) \dd x \,,
\end{align}
where $\mathbb{D} ( {\bf v} ) = \frac{1}{2} ( \nabla {\bf v} + \nabla^\top {\bf v} )$, and we have used \begin{equation*}
  \begin{aligned}
    ( {\bf b} \cdot \nabla ) {\bf a} \cdot {\bf b} = \frac12 {\bf b} \cdot ( \nabla {\bf a} + \nabla^\top {\bf a} ) \cdot {\bf b} \,,
  \end{aligned}
\end{equation*}
for the vectors ${\bf a}$ and ${\bf b}$.

Next we look at the last term on the left-hand side of \eqref{relative-energy-E1-2} for the two boundary cases \eqref{boundary1} and \eqref{boundary2}.

For the periodic case $\Omega = \mathbb{T}^3$, using Lemma \ref{Korn-type-ineq} in Appendix \ref{Sec:appendix} gives
\begin{equation}\label{u-v-H1-torus}
  \begin{aligned}
    \| {\bf u} - {\bf v} \|_{H^1} \leq c_1 \big ( \| \nabla ( {\bf u} - {\bf v} ) \|_{L^2} + \| \sqrt{\rho} ( {\bf u} - {\bf v} ) \|_{L^2} \big ) \,,
  \end{aligned}
\end{equation}
where $c_1 > 0$ is a constant dependent of $\gamma$, and $\gamma \geq \frac{6}{5}$ for $n = 3$ from Lemma \ref{Korn-type-ineq}.

For the bounded domain $\Omega$ with Dirichlet boundary condition, it needs the extra restriction ${\bf v}|_{\partial \Omega} = 0$. The  Poincar\'{e}'s inequality yields
\begin{equation}\label{u-v-H1-bdd}
  \begin{aligned}
    \| {\bf u} - {\bf v} \|_{H^1} \leq c_2 \| \nabla ( {\bf u} - {\bf v} ) \|_{L^2} \,,
  \end{aligned}
\end{equation}
where $c_2 > 0$ is a constant. Thus, we only need $\gamma > 1$ in this situation.

By integration by parts, one has
\begin{equation}\label{S(u-v)}
  \begin{aligned}
    \int \mathbb{S} (\nabla {\bf u} - \nabla {\bf v} ) : \nabla ( {\bf u} - {\bf v} ) \dd x = \mu \| \nabla ( {\bf u} - {\bf v} ) \|_{L^2}^2 + ( \mu + \lambda ) \| \dv ( {\bf u} - {\bf v} ) \|_{L^2}^2 \,.
  \end{aligned}
\end{equation}

Using Sobolev embedding $H^1 \hookrightarrow L^6$ with the generic constant $\hat{c}$, and the estimates \eqref{u-v-H1-torus}, \eqref{u-v-H1-bdd} and \eqref{S(u-v)}, we have
\begin{align}\label{u-v-L6-torus}
    \| {\bf v} - {\bf u} \|_{L^6}^2 \leq & \hat{c} \| {\bf v} - {\bf u} \|_{H^1}^2 \leq 2 \hat{c} c_{\gamma}^2 \| \nabla ( {\bf u} - {\bf v} ) \|_{L^2}^2 + 2 \hat{c} c_{\gamma}^2 \| \sqrt{\rho} ( {\bf u} - {\bf v} ) \|_{L^2}^2 \nonumber\\
    \leq & \frac{ 2 \hat{c} c_{\gamma}^2 }{ \mu } \int \mathbb{S} (\nabla {\bf u} - \nabla {\bf v} ) : \nabla ( {\bf u} - {\bf v} ) \dd x + 2 \hat{c} c_{\gamma}^2 \| \sqrt{\rho} ( {\bf u} - {\bf v} ) \|_{L^2}^2
\end{align}
for the periodic case \eqref{boundary1}, and
\begin{equation}\label{u-v-L6-bdd}
  \begin{aligned}
    \| {\bf v} - {\bf u} \|_{L^6}^2 \leq & \hat{c} \| {\bf v} - {\bf u} \|_{H^1}^2 \leq \frac{ 2 \hat{c} c_{\gamma}^2 }{ \mu } \int \mathbb{S} (\nabla {\bf u} - \nabla {\bf v} ) : \nabla ( {\bf u} - {\bf v} ) \dd x
  \end{aligned}
\end{equation}
for the Dirichlet boundary case \eqref{boundary2}. Here, $c_{\gamma} = \max \{ c_1, c_2 \}$.

In view of the convexity of ${\rm P}(\rho)$ with $\gamma > 1$ and $r \geq r_1 > 0$, it holds that
\begin{equation}\label{convexity-P}
  \begin{aligned}
    {\rm P}(\rho) - {\rm P}'(r) ( \rho - r ) - {\rm P}(r) \geq
    \left\{
    \begin{array}{ll}
    C ( 1 + |\rho - r|^2 ), \quad\, \frac{1}{2} r \leq \rho \leq \frac{3}{2} r \,, \\
    C ( 1 + |\rho - r|^\gamma ), \,\quad 0 \leq \rho < \frac{1}{2} r, \, \rho > \frac{3}{2} r \,.
    \end{array}\right.
  \end{aligned}
\end{equation}

For the convenience of notations, letting $f$ be a given function and the bold ${\bf 1}$ be characteristic function, we denote
\begin{align*}
    & \Omega_{\rm ess} = \Big \{ x \in \Omega : | \rho - r | \leq \frac12 r \Big \} \,, \quad \Omega_{\rm res} = \Big \{ x \in \Omega : 0 \leq \rho < \frac{1}{2} r \,, \; \rho > \frac{3}{2} r \Big \} \,, \\
    & f = [f]_{\rm ess} + [f]_{\rm res} \,, \quad [f]_{\rm ess} = f {\bf 1}_{ \Omega_{\rm ess} } \,, \quad [f]_{\rm res} = f {\bf 1}_{ \Omega_{\rm res} } \,.
\end{align*}

Utilizing the above notations, and by H\"{o}lder's inequality, one deduces that
\begin{align}\label{rSuv}
    & \int \frac{\rho - r}{r} \dv \mathbb{S} (\nabla {\bf v}) \cdot ( {\bf v} - {\bf u} ) \dd x \nonumber\\
    = & \int \frac{1}{r} ( \sqrt{\rho} - \sqrt{r} ) ( \sqrt{\rho} + \sqrt{r} ) \dv \mathbb{S} (\nabla {\bf v}) \cdot ( {\bf v} - {\bf u} ) \dd x \nonumber\\
    = & \int_{ \Omega_{\rm ess} + \Omega_{\rm res} } \frac{ \sqrt{\rho} - \sqrt{r} }{ \sqrt{r} } \dv \mathbb{S} (\nabla {\bf v}) \cdot ( {\bf v} - {\bf u} ) \dd x \nonumber\\
    & + \int_{ \Omega_{\rm ess} + \Omega_{\rm res} } \frac{ \sqrt{\rho} - \sqrt{r} }{r} \dv \mathbb{S} (\nabla {\bf v}) \cdot \sqrt{\rho} ( {\bf v} - {\bf u} ) \dd x \nonumber\\
    \leq & \int_{ \Omega_{\rm ess} + \Omega_{\rm res} } \frac{1}{ \sqrt{r} } | \rho - r |^{\frac12} | \dv \mathbb{S} (\nabla {\bf v}) | | {\bf v} - {\bf u} | \dd x \nonumber\\
    & + \int_{ \Omega_{\rm ess} + \Omega_{\rm res} } \frac{1}{r} | \rho - r |^{\frac12} | \dv \mathbb{S} (\nabla {\bf v}) | | \sqrt{\rho} ( {\bf v} - {\bf u} ) | \dd x \nonumber\\
    \leq & \frac{ \sqrt{2} }{2} \| \dv \mathbb{S} (\nabla {\bf v}) \|_{ L^{ \frac{6}{5} } } \| {\bf v} - {\bf u} \|_{ L^6 } \nonumber\\
    & + \Big \| \frac{1}{ \sqrt{r} } \Big \|_{L^\infty} \| [ \rho - r ]_{\rm res} \|_{ L^\gamma }^{ \frac12 } \| \dv \mathbb{S} (\nabla {\bf v}) \|_{ L^{ \frac{6\gamma}{5\gamma - 3} } } \| {\bf v} - {\bf u} \|_{ L^6 } \nonumber\\
    & + \Big \| \frac{1}{ \sqrt{2r} } \Big \|_{L^\infty} \| \dv \mathbb{S} (\nabla {\bf v}) \|_{ L^2 } \| \sqrt{\rho} ( {\bf v} - {\bf u} ) \|_{ L^2 } \nonumber\\
    & + \Big \| \frac{1}{r} \Big \|_{L^\infty} \| [ \rho - r ]_{\rm res} \|_{ L^\gamma }^{ \frac12 } \| \dv \mathbb{S} (\nabla {\bf v}) \|_{ L^{ \frac{2\gamma}{\gamma - 1} } } \| \sqrt{\rho} ( {\bf v} - {\bf u} ) \|_{ L^2 } \,,
\end{align}
where we have also employed the elementary inequality
\begin{equation*}
  \begin{aligned}
    | \rho^\theta - r^\theta | \leq |\rho - r|^\theta \,, \;\; 0 \leq \theta \leq 1 \,,
  \end{aligned}
\end{equation*}
with the special one $\theta = \frac12$. Noticing that $r \geq r_1 >0$, with the help of \eqref{u-v-L6-torus}, \eqref{u-v-L6-bdd}, \eqref{convexity-P} and Young's inequality, we have
\begin{align}\label{r-S-u-v-torus}
    & \int \frac{\rho - r}{r} \dv \mathbb{S} (\nabla {\bf v}) \cdot ( {\bf v} - {\bf u} ) \dd x \nonumber\\
    \leq & \frac{ \mu }{ 8 \hat{c} c_{\gamma}^2 } \| {\bf v} - {\bf u} \|_{ L^6 }^2 + \frac{ \hat{c} c_{\gamma}^2 }{ \mu } \| \dv \mathbb{S} (\nabla {\bf v}) \|_{ L^{ \frac{6}{5} } }^2 \nonumber\\
    & + \frac{ \mu }{ 8 \hat{c} c_{\gamma}^2 } \| {\bf v} - {\bf u} \|_{ L^6 }^2 + \frac{ 2 \hat{c} c_{\gamma}^2 }{ \mu r_1 } \| \dv \mathbb{S} (\nabla {\bf v}) \|_{ L^{ \frac{6\gamma}{5\gamma - 3} } }^2 \| [ \rho - r ]_{\rm res} \|_{ L^\gamma } \nonumber\\
    & + \frac{1}{ 2 \sqrt{2 r_1} } \| \dv \mathbb{S} (\nabla {\bf v}) \|_{ L^2 } ( 1 + \| \sqrt{\rho} ( {\bf v} - {\bf u} ) \|_{ L^2 }^2 ) \nonumber\\
    & + \frac{1}{r_1} \| \dv \mathbb{S} (\nabla {\bf v}) \|_{ L^{ \frac{2\gamma}{\gamma - 1} } } ( \| \sqrt{\rho} ( {\bf v} - {\bf u} ) \|_{ L^2 }^2 + \| [ \rho - r ]_{\rm res} \|_{ L^\gamma } ) \nonumber\\
    \leq & \frac{1}{2} \int \mathbb{S} (\nabla {\bf u} - \nabla {\bf v} ) : \nabla ( {\bf u} - {\bf v} ) \dd x + \frac{\mu}{2} \int \rho | {\bf u} - {\bf v} |^2 \dd x \nonumber\\
    & + C \frac{ \hat{c} c_{\gamma}^2 }{ \mu r_1 } \| \dv \mathbb{S} (\nabla {\bf v}) \|^2_{ L^{ \frac{6 \gamma}{5\gamma - 3} } } \int {\rm P}(\rho) - {\rm P}'(r)( \rho - r ) - {\rm P}(r) \dd x \nonumber\\
    & + C \frac{ 1 + \sqrt{r_1} }{r_1} \| \dv \mathbb{S} (\nabla {\bf v}) \|_{ L^{ \frac{2 \gamma}{\gamma - 1} } } \int \frac{1}{2} \rho | {\bf u} - {\bf v} |^2 + {\rm P}(\rho) - {\rm P}'(r)( \rho - r ) - {\rm P}(r) \dd x
\end{align}
under the case \eqref{boundary1}, and
\begin{align}\label{r-S-u-v-bdd}
    & \int \frac{\rho - r}{r} \dv \mathbb{S} (\nabla {\bf v}) \cdot ( {\bf v} - {\bf u} ) \dd x \nonumber\\
    \leq & \frac{1}{2} \int \mathbb{S} (\nabla {\bf u} - \nabla {\bf v} ) : \nabla ( {\bf u} - {\bf v} ) \dd x \nonumber\\
    & + C \frac{ \hat{c} c_{\gamma}^2 }{ \mu r_1 } \| \dv \mathbb{S} (\nabla {\bf v}) \|^2_{ L^{ \frac{6 \gamma}{5\gamma - 3} } } \int {\rm P}(\rho) - {\rm P}'(r)( \rho - r ) - {\rm P}(r) \dd x \nonumber\\
    & + C \frac{ 1 + \sqrt{r_1} }{r_1} \| \dv \mathbb{S} (\nabla {\bf v}) \|_{ L^{ \frac{2 \gamma}{\gamma - 1} } } \int \frac{1}{2} \rho | {\bf u} - {\bf v} |^2 + {\rm P}(\rho) - {\rm P}'(r)( \rho - r ) - {\rm P}(r) \dd x
\end{align}
under the case \eqref{boundary2}.

Putting \eqref{r-S-u-v-torus} and \eqref{r-S-u-v-bdd} into \eqref{relative-energy-E1-2}, respectively, it follows that
\begin{equation}\label{relative-energy-dt}
  \begin{aligned}
    & {\rm \frac{d}{dt}} \int \frac{1}{2} \rho | {\bf u} - {\bf v} |^2 + {\rm P}(\rho) - {\rm P}'(r)( \rho - r ) - {\rm P}(r) \dd x \\
    & + \frac{1}{2} \int \mathbb{S} (\nabla {\bf u} - \nabla {\bf v} ) : \nabla ( {\bf u} - {\bf v} ) \dd x \\
    \leq & C_0 \Lambda ( {\bf v} ) \int \frac{1}{2} \rho | {\bf u} - {\bf v} |^2 + {\rm P}(\rho) - {\rm P}'(r) (\rho - r) - {\rm P}(r) \dd x \\
    & + \int \big | ( r - \rho ) {\rm P}''(r) E_1 (r, {\bf v}) \big | \dd x + \int \Big | \frac{\rho}{r} E_2 (r, {\bf v}) \cdot ( {\bf v} - {\bf u} ) \Big | \dd x \,,
  \end{aligned}
\end{equation}
where $C_0 > 0$ is a generic constant and
\begin{equation}\label{Lam(v)}
  \begin{aligned}
    \Lambda ( {\bf v} ) =
    \left\{
    \begin{array}{ll}
    \big ( \mu + \Lambda_0 ( {\bf v} ) \big ) \,, \, & \textrm{ for the case } \eqref{boundary1} \,, \\
   \big ( \Lambda_0 ( {\bf v} ) \big ) \,, \, & \textrm{ for the case } \eqref{boundary2} \,,
    \end{array}\right.
  \end{aligned}
\end{equation}
with
\begin{equation*}
    \Lambda_0 ( {\bf v} ) = \| \mathbb{D} ( {\bf v} ) \|_{L^\infty} + \frac{ \hat{c} c_{\gamma}^2 }{ \mu r_1 } \| \dv \mathbb{S} (\nabla {\bf v}) \|^2_{ L^{ \frac{6 \gamma}{5\gamma - 3} } } + \frac{ 1 + \sqrt{r_1} }{r_1} \| \dv \mathbb{S} (\nabla {\bf v}) \|_{ L^{ \frac{2 \gamma}{\gamma - 1} } } \,.
\end{equation*}

Applying Gr\"{o}nwall's inequality to \eqref{relative-energy-dt}, we obtain that, for all $t \geq 0$,
\begin{align}\label{relative-energy-ineq}
    {\rm LS} (\rho, {\bf u}; r, {\bf v}) : = & \int \frac{1}{2} \rho | {\bf u} - {\bf v} |^2 + {\rm P}(\rho) - {\rm P}'(r)( \rho - r ) - {\rm P}(r) \dd x \nonumber\\
    & + \frac{1}{2} \int_0^t \int \mathbb{S} (\nabla {\bf u} - \nabla {\bf v} ) : \nabla ( {\bf u} - {\bf v} ) \dd x \dd s \nonumber\\
    \leq & \exp \Big ( \int_0^t C_0 \Lambda ( {\bf v} ) \dd s \Big ) \int \frac{1}{2} \rho_0 \Big | \frac{{\bf m}_0}{\rho_0} - {\bf v}_0 \Big |^2 + {\rm P}(\rho_0) - {\rm P}'(r_0) (\rho_0 - r_0) - {\rm P}(r_0) \dd x \nonumber\\
    & + \int_0^t \int \exp \Big ( \int_s^t C_0 \Lambda ( {\bf v} ) \dd \tau \Big ) \big | ( r - \rho ) {\rm P}''(r) E_1 (r, {\bf v}) \big | \dd x \dd s \nonumber\\
    & + \int_0^t \int \exp \Big ( \int_s^t C_0 \Lambda ( {\bf v} ) \dd \tau \Big ) \Big | \frac{\rho}{r} E_2 (r, {\bf v}) \cdot ( {\bf v} - {\bf u} ) \Big | \dd x \dd s \nonumber\\
    = : & {\rm RS} (\rho, {\bf u}; r, {\bf v}) \,,
\end{align}
where $( r_0, {\bf v}_0 ) = ( r, {\bf v} ) \big |_{t = 0}$.

With the above a priori estimates at hand, we are ready to define  the dissipative solutions to the compressible isentropic
Navier-Stokes \eqref{com-NS} in the following sense.

\begin{defn}\label{def-dissipative-sol}
Let $\rho \in L^\infty (0, T; L^\gamma) \cap C_w ( [0, T]; L^\gamma )$, $\sqrt{\rho} {\bf u} \in L^\infty (0, T; L^2)$ and ${\bf u} \in L^2(0, T; H^1)$ for any fixed $T > 0$, $( \rho, {\bf u} )$ is a dissipative solution of the problem \eqref{com-NS}-\eqref{boundary1} or \eqref{com-NS}, \eqref{initial} and \eqref{boundary2}, if \eqref{relative-energy-ineq} holds for $(r, {\bf v})$ satisfying
\begin{align}\label{ds-regular1}
\left\{
 \begin{array}{ll}
   r \in C([0, T]; L^\gamma) \,, \quad r \geq r_1 >0 \,, \quad {\bf v} \in C([0, T]; L^{ \frac{2\gamma}{\gamma - 1} }) \,, \\
   \mathbb{D} ( {\bf v} ) \in L^1(0,T; L^\infty) \,, \quad \nabla {\bf v} \in L^2 (0, T; L^2) \,, \\
   \dv \mathbb{S} (\nabla {\bf v}) \in L^2(0,T; L^{ \frac{6 \gamma}{5\gamma - 3} } ) \cap L^1 (0,T; L^{ \frac{2 \gamma}{\gamma - 1} }) \,, \\
   {\rm P}'' (r) E_1(r, {\bf v}) \in L^1(0,T; L^{ \frac{\gamma}{\gamma - 1} } ) \,, \quad \frac{1}{r} E_2(r, {\bf v}) \in L^1(0,T; L^{ \frac{2\gamma}{\gamma - 1} }) \,,
 \end{array}\right.
\end{align}
where $r_1$ is a positive constant. Note that it needs the extra condition ${\bf v} |_{\partial \Omega} = 0$ for the bounded domain case \eqref{boundary2}.
\end{defn}

\begin{rem}\label{grad-v-regularity}
 If $(r,\bf v)$ satisfies \eqref{ds-regular1}, then  \eqref{relative-energy-ineq} is well-defined.
We observe that ${\bf v} \in L^1 (0, T; W^{1,p})$ for any $1 < p \leq \infty$ from \eqref{ds-regular1}. Indeed, in view of the Korn's inequality ${\rm(}$\cite{FN-book-2017}, Theorem 11.21${\rm )}$
\begin{equation*}
  \begin{aligned}
    \| {\bf v} \|_{ W^{1,p} } \leq C \Big ( \| \mathbb{D} ( {\bf v} ) \|_{L^p} + \int | {\bf v} | \dd x \Big ) \,, \;\; 1 < p < \infty \,,
  \end{aligned}
\end{equation*}
it follows that
\begin{equation*}
  \begin{aligned}
    \| {\bf v} \|_{ L^1 (0, T; W^{1,p}) } \leq  C \big ( \| \mathbb{D} ( {\bf v} ) \|_{ L^1 (0, T; L^\infty) } + \| {\bf v} \|_{ C([0, T]; L^{\frac{2\gamma}{\gamma - 1}}) } \big ) \,, \;\; 1 < p < \infty \,.
  \end{aligned}
\end{equation*}
For the case $p = \infty$, notice that on the right-hand side above is independent of $p$, by Lemma \ref{Lp-limit-L-infty}, then $\| {\bf v} \|_{ L^1 (0, T; W^{1,p}) } \rightarrow \| {\bf v} \|_{ L^1 (0, T; W^{1,\infty}) }$ as $p \rightarrow \infty$. Furthermore, it infers that ${\bf v} \cdot \nabla {\bf v} \in L^1 (0, T; L^{ \frac{2\gamma}{\gamma-1} })$.

We also find that $\nabla {\bf v} \in L^2 (0, T; L^6 \cup L^{ \frac{2 \gamma}{\gamma - 1} })$ from \eqref{ds-regular1}. Indeed, by the elliptic regularity result ${\rm (}$see \emph{(50)} in page \emph{389} of \cite{CK-jde-2006}${\rm )}$:
\begin{equation*}
  \begin{aligned}
    \| D^2 {\bf v} \|_{ L^p } \leq  C \big ( \| \mathcal{L} {\bf v} \|_{ L^p } + \| \nabla {\bf v} \|_{ L^p } \big ) \,, \;\; 1 < p < \infty \,,
  \end{aligned}
\end{equation*}
where $\mathcal{L} {\bf v} = - \mu \Delta {\bf v} - (\mu + \lambda) \nabla \dv {\bf v} = - \dv \mathbb{S} ( \nabla {\bf v} )$, together with $\dv \mathbb{S} ( \nabla {\bf v} ) \in L^2 (0, T; L^{ \frac{6 \gamma}{5 \gamma - 3} })$ and $\nabla {\bf v} \in L^2 (0, T; L^2)$ in \eqref{ds-regular1}, then, $D^2 {\bf v} \in L^2 (0, T; L^2)$ for $1 < \gamma < \frac{3}{2}$ and $D^2 {\bf v} \in L^2 (0, T; L^{ \frac{6 \gamma}{5\gamma - 3} })$ for $\gamma \geq \frac{3}{2}$. By the Sobolev embedding $H^1 \hookrightarrow L^6$ for the case $1 < \gamma < \frac{3}{2}$ and $W^{1, \frac{6 \gamma}{5 \gamma - 3}} \hookrightarrow L^q (1 \leq q \leq \frac{2 \gamma}{\gamma - 1})$ for the case $\gamma \geq \frac{3}{2}$, it follows that $\nabla {\bf v} \in L^2 (0, T; L^6 \cup L^{ \frac{2 \gamma}{\gamma - 1} })$.
\end{rem}

\begin{rem}
We observe that the conditions ${\rm P}'' (r) E_1(r, {\bf v}) \in L^1 (0,T; L^{ \frac{\gamma}{\gamma - 1} })$ and $\frac{1}{r} E_2(r, {\bf v}) \in L^1 (0, T; L^{ \frac{2\gamma}{\gamma - 1} })$ are equivalent to $\partial_t {\rm P}' (r) \in L^1 (0,T; L^{ \frac{\gamma}{\gamma - 1} })$ and $\nabla {\rm P}' (r) \in L^1 (0,T; L^{ \frac{2 \gamma}{\gamma - 1} })$ if $\partial_t {\bf v} \in L^1 (0,T; L^{ \frac{2 \gamma}{\gamma - 1} })$. Indeed, $\mathbb{D} ( {\bf v} ) \in L^1 (0, T; L^\infty)$ implies $\dv {\bf v} \in L^1 (0, T; L^\infty)$. Notice that
\begin{equation*}
  \begin{aligned}
    {\rm P}'' (r) E_1(r, {\bf v}) = \partial_t {\rm P}' (r) + {\rm p}'( r) \dv {\bf v} + {\bf v} \cdot \nabla {\rm P}' (r) \,, \\
    \frac{1}{r} E_2 ( r, {\bf v} ) = \partial_t {\bf v} + {\bf v} \cdot \nabla {\bf v} + \nabla {\rm P}' (r) - \frac{1}{r} \dv \mathbb{S} ( \nabla {\bf v} ) \,,
  \end{aligned}
\end{equation*}
along with $ r \in C (0,T; L^\gamma)$, $r \geq r_1 >0$, ${\bf v} \in C([0, T]; L^{ \frac{2\gamma}{\gamma - 1} })$, ${\bf v} \cdot \nabla {\bf v} \in L^1 (0, T; L^{ \frac{2\gamma}{\gamma-1} })$ and $\dv \mathbb{S} (\nabla {\bf v}) \in L^1 (0,T; L^{ \frac{2 \gamma}{\gamma - 1} })$, it infers the observation.
\end{rem}

Here we address our first main result on the existence of dissipative solution in the sense of Definition \ref{def-dissipative-sol}.

\begin{thm}\label{thm-existence-dissipative-sol}
Suppose that the assumptions of Definition \emph{\ref{def-dissipative-sol}} hold. In addition, we assume that
\begin{align}\label{ds-regular2}
\left\{
 \begin{array}{ll}
   r \leq r_2 \textrm{ if } \gamma > 2 \,, \quad  \partial_t r \in L^1 (0, T; L^{ \frac{\gamma}{\gamma - 1} }) \cap L^2 (0, T; L^{ \frac{2 \gamma}{\gamma + 1} } \cap L^{ \frac{6}{5} }) \,, \\
   \nabla r \in L^\infty (0, T; L^{ \frac{2 \gamma}{\gamma - 1} }) \,, \quad \partial_t {\bf v} \in L^1 (0, T; L^{ \frac{2 \gamma}{\gamma - 1} }) \cap L^2 (0, T; L^{ \frac{2 \gamma}{\gamma + 1} } \cap L^{ \frac{6}{5} } ) \,, \\
   \nabla {\bf v} \in L^2 (0, T; L^\infty) \,,
 \end{array}\right.
\end{align}
where $r_2$ is a positive constant. \\
1. For the periodic case \eqref{boundary1}: Let $\gamma \geq \frac{6}{5}$. Assume that the initial data \eqref{initial} satisfy
\begin{equation}\label{periodic-initial}
  \begin{aligned}
    \rho_0 \in L^\gamma \,, \;\; \rho_0 \geq 0 \,, \;\; \int \rho_0 \dd x \geq C_{\rho_0}\,, \;\; {\bf m}_0 = 0 \textrm{ a.e. in } \{ x \in \Omega : \rho_0 = 0 \} \,, \;\; \frac{ {\bf m}_0^2 }{ \rho_0 } \in L^1 \,,
  \end{aligned}
\end{equation}
for some positive constant $C_{\rho_0}$. Then, there exists a dissipative solution of the compressible isentropic Navier-Stokes system \eqref{com-NS} with \eqref{initial} and \eqref{boundary1}. \\
2. For the Dirichlet boundary case \eqref{boundary2}: Let $\gamma > 1$. Assume the initial data \eqref{initial} satisfy
\begin{equation}\label{Dirichlet-initial}
  \begin{aligned}
    \rho_0 \in L^\gamma \,, \;\; \rho_0 \geq 0 \,, \;\; {\bf m}_0 = 0 \textrm{ a.e. in } \{ x \in \Omega : \rho_0 = 0 \} \,, \;\; \frac{ {\bf m}_0^2 }{ \rho_0 } \in L^1 \,.
  \end{aligned}
\end{equation}
Then, there exists a dissipative solution of the compressible isentropic Navier-Stokes system \eqref{com-NS} with \eqref{initial} and \eqref{boundary2}.
\end{thm}


\begin{rem}
Feireisl, Novotn\'{y} and Sun \cite{FNS-indiana-2011} introduced a class of suitable weak solutions to the compressible barotropic Navier-Stokes equations and proved that the solution satisfies the relative entropy inequality 
 for $\gamma > \frac{3}{2}$. Compared with \cite{FNS-indiana-2011}, we define the dissipative solution satisfying the inequality \eqref{relative-energy-ineq}. We give the process of the density arguments at Section \ref{Sec:regular}. In addition, the adiabatic exponent can reduce to $\gamma \geq \frac{6}{5}$ for the periodic case and $\gamma > 1$ for the Dirichlet boundary case.
\end{rem}

Applying Theorem \ref{thm-existence-dissipative-sol} directly, we have the following result on the relationship of the strong solution and dissipative solution.

\begin{cor}\label{weak-strong-ds}
Assume that $(r, {\bf v})$ is a strong solution of \eqref{com-NS} and the initial data $(r, {\bf v})|_{t = 0} = (r_0, {\bf v}_0)$ together with \eqref{boundary1} or \eqref{boundary2}, satisfying the regularities
\begin{align}\label{regular-r-v}
\left\{
 \begin{array}{ll}
   r \in C([0,T]; L^\gamma) \,, \;\; r_1 \leq r \leq r_2 \,, \;\; \partial_t r \in L^1 (0, T; L^{ \frac{\gamma}{\gamma - 1} }) \cap L^2 (0, T; L^{ \frac{2 \gamma}{\gamma + 1} } \cap L^{ \frac{6}{5} }) \,, \\
   {\bf v} \in C([0,T]; L^{ \frac{2\gamma}{\gamma - 1} }) \,,  \;\; \partial_t {\bf v} \in L^1 (0, T; L^{ \frac{2 \gamma}{\gamma - 1} }) \cap L^2 (0, T; L^{ \frac{2 \gamma}{\gamma + 1} } \cap L^{ \frac{6}{5} }) \,, \\
   \nabla r \in L^\infty (0,T; L^{ \frac{2\gamma}{\gamma - 1} }) \,, \;\; \dv \mathbb{S} (\nabla {\bf v}) \in L^2(0,T; L^{ \frac{6 \gamma}{5\gamma - 3} } ) \cap L^1 (0,T; L^{ \frac{2 \gamma}{\gamma - 1} }) \,, \\
   \nabla {\bf v} \in L^2 (0, T; L^\infty) \,,
 \end{array}\right.
\end{align}
for some positive constants $r_1$ and $r_2$, and any $T \in (0, T_{\rm max})$, where $T_{\rm max}$ is the maximal existence time. Let $(\rho, {\bf u})$ be a dissipative solution of the system \eqref{com-NS}-\eqref{initial} together with \eqref{boundary1} or \eqref{boundary2}, and the initial data satisfy
\begin{equation*}
  \begin{aligned}
    \int \frac{1}{2} \rho_0 \Big | \frac{{\bf m}_0}{\rho_0} - {\bf v}_0 \Big |^2 + {\rm P}(\rho_0) - {\rm P}'(r_0) (\rho_0 - r_0) - {\rm P}(r_0) \dd x = 0 \,.
  \end{aligned}
\end{equation*}
Then, the dissipative solution $(\rho, {\bf u})$ is equal to $(r, {\bf v})$ on a.e. $(t, x) \in [0, T] \times \Omega$.
\end{cor}
\begin{rem}
The condition $\nabla {\bf v} \in L^2 (0, T; L^\infty)$ implies that $\mathbb{D} ( {\bf v} ) = \frac12 ( \nabla {\bf v} + \nabla^\top {\bf v} ) \in L^1(0,T; L^\infty)$, which is corresponding to the blow up criteria of the Navier-Stokes equations given by Huang et al. \cite{HLX-cmp-2011}. Under the condition $\mathbb{D} ( {\bf v} ) \in L^1(0,T; L^\infty)$, the regularities \eqref{regular-r-v} can be replaced by making restriction on the initial data and $\gamma$ in some situations. For example, endowing the initial data
\begin{equation}\label{initial-str}
  \begin{aligned}
    & 0 < \underline{r} \leq r_0 \leq \overline{r} \,, \quad r_0 \in W^{1,p} \, \; \textrm{ for } \, \; p > 6 \,, \quad {\bf v}_0 \in H^2 \,,
  \end{aligned}
\end{equation}
for some positive constants $\underline{r}$ and $\overline{r}$, and making use of the condition $\mathbb{D} ( {\bf v} ) \in L^1(0,T; L^\infty)$, a direct conclusion from \cite{HLX-cmp-2011} shows that there exists a global strong solution $(r, {\bf v})$ with
\begin{align}\label{regular-r-v-1}
\left\{
 \begin{array}{ll}
   r_1 \leq r \leq r_2 \,, \quad r \in C ([0,T); W^{1,6}) \,, \;\; \partial_t r \in C([0,T); L^6) \,, \\
   {\bf v} \in C ([0,T); H^2) \cap L^2 (0,T; W^{2,6}) \,, \;\; \partial_t {\bf v} \in L^\infty (0,T; L^2) \cap L^2 (0,T; H^1) \,,
 \end{array}\right.
\end{align}
for any $T \in (0, \infty)$. We see that $(r, {\bf v})$ with \eqref{regular-r-v-1} meets the requirement of \eqref{regular-r-v} for $\gamma \geq \frac{3}{2}$. It says that \eqref{regular-r-v} can be substituted by \eqref{initial-str} for $\gamma \geq \frac{3}{2}$. 
\end{rem}

Our next goal is to show the weak solution of the compressible isentropic Navier-Stokes equations is also a dissipative solution. We first recall the definition of weak solution as follows.

\begin{defn}\label{def-weak-sol}
A pair $(\rho, {\bf u})$ is a weak solution of the problem \eqref{com-NS}-\eqref{boundary1} or \eqref{com-NS}, \eqref{initial} and \eqref{boundary2} provided that, for any fixed $T > 0$,
\begin{equation}\label{weak-est}
  \begin{aligned}
    \rho \in L^\infty (0,T; L^\gamma) \,, \quad \sqrt{\rho} {\bf u} \in L^\infty (0,T; L^2) \,, \quad {\bf u} \in L^2 (0,T; H^1) \,,
  \end{aligned}
\end{equation}
and the continuity equation $\eqref{com-NS}_1$ and the momentum equation $\eqref{com-NS}_2$ are satisfied in $\mathscr{D}' ( [0,T] \times \Omega )$, that is, for $\Psi \in C_c^\infty ( [0,T] \times \Omega )$, $\Phi \in C_c^\infty ( [0,T] \times \Omega )$,
\begin{align}
    & \int \rho (T,x) \Psi (T,x) \dd x - \int \rho_0 \Psi (0,x) \dd x = \int_0^T \int \rho \partial_t \Psi + \rho {\bf u} \cdot \nabla \Psi \dd x \dd t \,, \label{weak-conti-D}\\
    & \int \rho (T,x) {\bf u} (T,x) \cdot \Phi (T,x) \dd x - \int {\bf m}_0 \cdot \Phi (0,x) \dd x = \int_0^T \int \rho {\bf u} \cdot \partial_t \Phi + \rho {\bf u} \otimes {\bf u} : \nabla \Phi \nonumber \\
    & \qquad\qquad\qquad\qquad\qquad\qquad\qquad\quad\quad\quad\quad + A \rho^\gamma \dv \Phi - \mathbb{S}( \nabla {\bf u} ) : \nabla \Phi \dd x \dd t \,, \label{weak-mome-D}
\end{align}
and the energy inequality holds
  \begin{equation}\label{ws-weak-energy-ineq}
  \begin{aligned}
    \int \frac12 \rho |{\bf u}|^2 + \frac{A}{\gamma-1} \rho^\gamma \dd x + \int_0^t \int \mathbb{S}( \nabla {\bf u} ) : \nabla {\bf u} \dd x \dd t \leq \int \frac12 \frac{| {\bf m}_0 |^2}{\rho_0} + \frac{A}{\gamma-1} \rho_0^\gamma \dd x \,,
  \end{aligned}
\end{equation}
for almost every $t \in [0,T]$.
\end{defn}

Then, we write our next theorem in the following.

\begin{thm}\label{thm-weak-dissipative}
 Suppose that the assumptions of Theorem \ref{thm-existence-dissipative-sol} hold. If $(\rho, {\bf u})$ is a weak solution of the problem \eqref{com-NS}-\eqref{boundary1} or \eqref{com-NS}, \eqref{initial} and \eqref{boundary2}, then, the weak solution is a dissipative solution in the sense of Definition \ref{def-dissipative-sol}.
\end{thm}

The weak solution of the compressible isentropic Navier-Stokes equations also has the weak-strong uniqueness property, see \cite{FJN-jmfm-2012}. By means of Theorem \ref{thm-weak-dissipative}, it can directly give another version proof of the weak-strong uniqueness property for the weak solution of the compressible isentropic Navier-Stokes equations. We state the conclusion by a corollary as follows.

\begin{cor}\label{weak-strong-ws}
Under the same assumptions of Corollary \ref{weak-strong-ds}, if we assume that $(\rho, {\bf u})$ is a weak solution of the problem \eqref{com-NS}-\eqref{boundary1} or \eqref{com-NS}, \eqref{initial} and \eqref{boundary2}, then $( \rho, {\bf u} ) = ( r, {\bf v} )$ on a.e. $(t,x) \in [0,T] \times \Omega$.
\end{cor}

\smallskip

\section{ regularization of $r$ and ${\bf v}$ }\label{Sec:regular}
This section devotes to show that the smooth function $r$ and $\bf v$ can be replaced by  the functions with regularities given in \eqref{ds-regular1} and \eqref{ds-regular2}. Thus, the following is main result in this section.

\begin{prop}

If \eqref{relative-energy-ineq} holds for smooth functions $r$ and ${\bf v}$, then it also holds for the functions  satisfying  \eqref{ds-regular1} and \eqref{ds-regular2}.

\end{prop}
\begin{proof}

We first extend $r$ and ${\bf v}$ on $[ - \delta_0, T + \delta_0 ] \times \widetilde{\Omega}$ with small $\delta_0 >0$ and $\Omega \subset\subset \widetilde{\Omega}$, which still satisfy the regularities \eqref{ds-regular1} and \eqref{ds-regular2}. Let
\begin{align}
    & \eta (x) \in C_c^\infty ( \mathbb{R}^3 ) \,, \;\; {\rm supp} \ \eta = \{ x \in \mathbb{R}^3 : |x| \leq 1 \} \,, \;\; \int_{\mathbb{R}^3} \eta (x) \dd x =1 \,, \;\; \eta^\delta ( x) = \frac{1}{\delta^3} \eta( \frac{x}{\delta} ) \,, \label{mollifier1} \\
    & \tilde{\eta} (t) \in C_c^\infty ( \mathbb{R} ) \,, \;\; {\rm supp} \ \tilde{\eta} = \{ t \in \mathbb{R} : |t| \leq 1 \} \,, \;\; \int_{\mathbb{R}} \tilde{\eta} (t) \dd t =1 \,, \;\; \tilde{\eta}^\delta ( t ) = \frac{1}{\delta} \tilde{\eta}( \frac{t}{\delta} ) \,, \label{mollifier2}
\end{align}
where $0 < \delta \leq 1$. We mollify $r$ and ${\bf v}$ with respect to space and time in the following way:
\begin{equation*}
  \begin{aligned}
    & r_{t,x}^\delta (t, x) = ( r \ast \eta^\delta ) \ast \tilde{\eta}^\delta (t, x) = \int_{ \mathbb{R} } \int_{ \mathbb{R}^n } r (t - s, x - y) \eta^\delta (y) \dd y \tilde{\eta}^\delta (s) \dd s \,, \\
    & {\bf v}_{t,x}^\delta (t, x) = ( {\bf v} \ast \eta^\delta ) \ast \tilde{\eta}^\delta (t, x) = \int_{ \mathbb{R} } \int_{ \mathbb{R}^n } {\bf v} (t - s, x - y) \eta^\delta (y) \dd y \tilde{\eta}^\delta (s) \dd s \,.
  \end{aligned}
\end{equation*}

Since
\begin{align*}
   & \nabla r \in L^1 (- \delta_0, T + \delta_0; L^{ \frac{2 \gamma}{\gamma - 1} } ( \widetilde{\Omega} )) \,, \quad \partial_t r \in L^1 (- \delta_0, T + \delta_0; L^{ \frac{\gamma}{\gamma - 1} } ( \widetilde{\Omega} )) \,, \\
   & \nabla {\bf v} \in L^2 (- \delta_0, T + \delta_0; L^2 ( \widetilde{\Omega} )) \,, \quad \dv {\bf v} \in L^2 (- \delta_0, T + \delta_0; L^2 ( \widetilde{\Omega} )) \,,
\end{align*}
by Lemmas \ref{mollifier-derivative} and \ref{mollifier-Lp}, then, as $\delta \rightarrow 0$,
\begin{align}
    & \| \nabla r_{t,x}^\delta - \nabla r \|_{L^1 (0, T; L^{\frac{2 \gamma}{\gamma - 1}} )} = \| (\nabla r )_{t,x}^\delta - \nabla r \|_{L^1 (0, T; L^{\frac{2 \gamma}{\gamma - 1}} )} \rightarrow 0 \,, \label{grad-r-delta} \\
    & \| \partial_t r_{t,x}^\delta - \partial_t r \|_{L^1 (0, T; L^{\frac{\gamma}{\gamma - 1}} )} = \| (\partial_t r )_{t,x}^\delta - \partial_t r \|_{L^1 (0, T; L^{\frac{\gamma}{\gamma - 1}} )} \rightarrow 0 \,, \label{dt-r-delta} \\
    & \| \nabla {\bf v}_{t,x}^\delta - \nabla {\bf v} \|_{L^2(0, T; L^2)} = \| ( \nabla {\bf v} )_{t,x}^\delta - \nabla {\bf v} \|_{L^2(0, T; L^2)} \rightarrow 0 \,, \label{grad-v-delta} \\
    & \| \dv {\bf v}_{t,x}^\delta - \dv {\bf v} \|_{L^2(0, T; L^2)} = \| ( \dv {\bf v} )_{t,x}^\delta - \dv {\bf v} \|_{L^2(0, T; L^2)} \rightarrow 0 \,. \label{div-v-delta}
\end{align}

For ${\bf v} \in C ([- \delta_0, T + \delta_0]; L^{\frac{2 \gamma}{\gamma - 1}} ( \widetilde{\Omega} ))$, $\partial_t {\bf v} \in L^1 (- \delta_0, T + \delta_0; L^{ \frac{2 \gamma}{\gamma - 1} } ( \widetilde{\Omega} ))$, by Lemma \ref{mollifier-L-infty}, then, as $\delta \rightarrow 0$,
\begin{align}\label{v-delta}
    & \| {\bf v}_{t,x}^\delta - {\bf v} \|_{L^\infty (0, T; L^{\frac{2 \gamma}{\gamma - 1}})} \rightarrow 0 \,.
\end{align}

For $\partial_t r \in L^1 (- \delta_0, T + \delta_0; L^{ \frac{\gamma}{\gamma - 1} } ( \widetilde{\Omega} ))$ and $\nabla r \in L^\infty (- \delta_0, T + \delta_0; L^{\frac{2 \gamma}{\gamma - 1}} ( \widetilde{\Omega} ))$, by Lemma \ref{mollifier-L-infty}, we know that, as $\delta \rightarrow 0$,
\begin{align}\label{r-delta-infty}
    \| r_{t,x}^\delta - r \|_{ L^\infty (0, T; L^\infty) } \rightarrow 0 \,.
\end{align}
In addition, since $r \geq r_1$ if $1 < \gamma \leq 2$ and $r_1 \leq r \leq r_2$ if $\gamma > 2$, and noticing that $\int_{ \mathbb{R}^3 } \eta^\delta (y) \dd y = 1$ and $\int_{ \mathbb{R} } \tilde{ \eta }^\delta (s) \dd s = 1$, then we have
\begin{equation}\label{r-delta-bdd}
  \begin{aligned}
    r_{t,x}^\delta \geq r_1 \; \textrm{ if } \; 1 < \gamma \leq 2 \,, \quad r_1 \leq r_{t,x}^\delta (t, x) \leq r_2 \; \textrm{ if } \; \gamma > 2 \,.
  \end{aligned}
\end{equation}

By Lagrange mean value theorem, there exists a $\theta_1 \in (0, 1)$ such that
\begin{equation*}
  \begin{aligned}
    {\rm P}'' ( r_{t,x}^\delta ) - {\rm P}'' (r) = & {\rm P}''' ( \theta_1 r_{t,x}^\delta + ( 1 - \theta_1 ) r ) ( r_{t,x}^\delta - r ) \\
    = & A \gamma ( \gamma - 2 ) ( \gamma - 3 ) [ \theta_1 r_{t,x}^\delta + ( 1 - \theta_1 ) r ]^{ \gamma - 3 } ( r_{t,x}^\delta - r ) \,.
  \end{aligned}
\end{equation*}
Combining \eqref{r-delta-infty} and \eqref{r-delta-bdd}, it follows from the above equality that, as $\delta \rightarrow 0$,
\begin{equation}\label{P''(r)}
  \begin{aligned}
    \| {\rm P}'' ( r_{t,x}^\delta ) - {\rm P}'' (r) \|_{ L^\infty (0, T; L^\infty) }  \rightarrow 0 \,.
  \end{aligned}
\end{equation}
By the same argument, we have
\begin{equation}\label{p'(r)}
  \begin{aligned}
    \| {\rm p}' ( r_{t,x}^\delta ) - {\rm p}' (r) \|_{ L^\infty (0, T; L^{ \frac{\gamma}{\gamma - 1} }) }  \rightarrow 0 \,.
  \end{aligned}
\end{equation}

Since
\begin{equation*}
  \begin{aligned}
    \nabla {\rm P}' ( r_{t,x}^\delta ) - \nabla {\rm P}' (r) = {\rm P}''( r_{t,x}^\delta ) ( \nabla r_{t,x}^\delta - \nabla r ) +  [ {\rm P}''( r_{t,x}^\delta ) - {\rm P}''(r) ] \nabla r \,,
  \end{aligned}
\end{equation*}
along with \eqref{grad-r-delta}, \eqref{r-delta-bdd} and \eqref{P''(r)}, then it holds, as $\delta \rightarrow 0$,
\begin{equation}\label{grad-P'(r)}
  \begin{aligned}
    \| \nabla {\rm P}' ( r_{t,x}^\delta ) - \nabla {\rm P}' (r) \|_{ L^1 (0, T; L^{ \frac{2\gamma}{\gamma-1} } ) } \rightarrow 0 \,.
  \end{aligned}
\end{equation}

Next, we will show that ${\rm P}'' (r) E_1 (r, {\bf v})$ can be approximated by ${\rm P}'' ( r_{t,x}^\delta ) E_1 ( r_{t,x}^\delta, {\bf v}_{t,x}^\delta )$. Recall the definition $E_1 (r, {\bf v})$ in \eqref{E1-2} and the relation $r {\rm P}'' (r) = {\rm p}' (r)$, then,
\begin{align}\label{P''-E1}
    {\rm P}'' ( r_{t,x}^\delta ) E_1 (r_{t,x}^\delta, {\bf v}_{t,x}^\delta) = & \partial_t {\rm P}' ( r_{t,x}^\delta ) + {\bf v}_{t,x}^\delta \cdot \nabla {\rm P}'' ( r_{t,x}^\delta ) + {\rm p}' (r_{t,x}^\delta ) \dv {\bf v}_{t,x}^\delta \nonumber\\
    = & ( {\rm P}'' (r) E_1 (r, {\bf v}) )_{t,x}^\delta + \partial_t {\rm P}' ( r_{t,x}^\delta ) - ( \partial_t {\rm P}' (r) )_{t,x}^\delta + {\bf v}_{t,x}^\delta \cdot \nabla {\rm P}' ( r_{t,x}^\delta ) \nonumber\\
    & - ( {\bf v} \cdot \nabla {\rm P}' (r) )_{t,x}^\delta + {\rm p}' ( r_{t,x}^\delta ) \dv {\bf v}_{t,x}^\delta - ( {\rm p}' (r) \dv {\bf v} )_{t,x}^\delta \,.
\end{align}

By Lemma \ref{mollifier-Lp}, we have, as $\delta \rightarrow 0$,
\begin{equation}\label{P''E1-delta}
  \begin{aligned}
    ( {\rm P}'' (r) E_1 (r, {\bf v}) )_{t,x}^\delta \rightarrow {\rm P}'' (r) E_1 (r, {\bf v}) \ \ \textrm{ strongly in } \; L^1 (0, T; L^{ \frac{\gamma}{\gamma-1} } ) \,.
  \end{aligned}
\end{equation}

For the term $\partial_t {\rm P}' ( r_{t,x}^\delta ) - ( \partial_t {\rm P}' (r) )_{t,x}^\delta$, it can be rewritten as
\begin{align}\label{dt-P'(r)-minus}
    \partial_t {\rm P}' ( r_{t,x}^\delta ) - ( \partial_t {\rm P}' (r) )_{t,x}^\delta = & {\rm P}'' (r) ( \partial_t r_{t,x}^\delta - \partial_t r ) + [ {\rm P}'' ( r_{t,x}^\delta ) - {\rm P}'' (r) ] \partial_t r_{t,x}^\delta \nonumber\\
    & + \partial_t {\rm P}' (r) - ( \partial_t {\rm P}' (r) )_{t,x}^\delta \,.
\end{align}
By \eqref{dt-r-delta}, \eqref{P''(r)} and Lemma \ref{mollifier-Lp},we have, as $\delta \rightarrow 0$,
\begin{equation}\label{dt-P'(r)-delta}
  \begin{aligned}
    \partial_t {\rm P}' ( r_{t,x}^\delta ) - ( \partial_t {\rm P}' (r) )_{t,x}^\delta \rightarrow 0 \ \ \textrm{ in } \; L^1 (0, T; L^{ \frac{\gamma}{\gamma-1} } ) \,.
  \end{aligned}
\end{equation}
Similarly, we have
\begin{equation}\label{v-grad-P'-delta}
  \begin{aligned}
    {\bf v}_{t,x}^\delta \cdot \nabla {\rm P}' ( r_{t,x}^\delta ) - ( {\bf v} \cdot \nabla {\rm P}' (r) )_{t,x}^\delta \rightarrow 0 \ \ \textrm{ in } \; L^1 (0, T; L^{ \frac{\gamma}{\gamma-1} } ) \,.
  \end{aligned}
\end{equation}

Now, we turn to deal with the term ${\rm p}' ( r_{t,x}^\delta ) \dv {\bf v}_{t,x}^\delta - ( {\rm p}' (r) \dv {\bf v} )_{t,x}^\delta$, which can be rewritten as
\begin{align*}
    & {\rm p}' ( r_{t,x}^\delta ) \dv {\bf v}_{t,x}^\delta - ( {\rm p}' (r) \dv {\bf v} )_{t,x}^\delta \\
    = & \Big ( {\rm p}' ( r_{t,x}^\delta ) - [ {\rm p}'(r) ]_{t,x}^\delta \Big ) \dv {\bf v}_{t,x}^\delta + [ {\rm p}'(r) ]_{t,x}^\delta \dv {\bf v}_{t,x}^\delta - ( {\rm p}' (r) \dv {\bf v} )_{t,x}^\delta \\
    = & \Big ( {\rm p}' ( r_{t,x}^\delta ) - [ {\rm p}'(r) ]_{t,x}^\delta \Big ) \dv {\bf v}_{t,x}^\delta + \Big ( [ {\rm p}'(r) ]_{t,x}^\delta - {\rm p}'(r) \Big ) ( \dv {\bf v}_{t,x}^\delta - \dv {\bf v} ) \\
    & - \int_{ \mathbb{R} } \int_{ \mathbb{R}^n } \Big ( [ {\rm p}' (r) ] (t, x) - [ {\rm p}'(r) ] (t-s, x-y) \Big ) \dv ( {\bf v}(t, x) \\
    & \quad \quad \quad\quad - {\bf v} (t-s, x-y) ) \eta^\delta (y) \dd y \tilde{ \eta }^\delta (s) \dd s \\
    = & \Big ( {\rm p}' ( r_{t,x}^\delta ) - {\rm p}'(r) \Big ) \dv {\bf v}_{t,x}^\delta - \Big ( [ {\rm p}'(r) ]_{t,x}^\delta - {\rm p}'(r) \Big ) \dv {\bf v} \\
    & - \int_{ \mathbb{R} } \int_{ \mathbb{R}^n } \Big ( [ {\rm p}' (r) ] (t, x) - [ {\rm p}'(r) ] (t - \delta s, x - \delta y) \Big ) \dv ( {\bf v}(t, x) \\
    & \quad \quad \quad\quad - {\bf v} (t - \delta s, x - \delta y) ) \eta (y) \dd y \tilde{ \eta } (s) \dd s \\
    = : & I_1 + I_2 + I_3 \,.
\end{align*}
For the term $I_1$, by \eqref{p'(r)}, we get, as $\delta \rightarrow 0$,
\begin{equation*}
  \begin{aligned}
    \| I_1 \|_{ L^1 (0, T; L^{ \frac{\gamma}{\gamma-1} } ) } \leq C \| \dv {\bf v} \|_{ L^1 (0, T; L^\infty) } \| {\rm p}'(r_{t,x}^\delta) - {\rm p}'(r) \|_{ L^\infty (0, T; L^{ \frac{\gamma}{\gamma - 1} }) } \rightarrow 0 \,.
  \end{aligned}
\end{equation*}
For the term $I_2$, thanks to ${\rm p}'(r)= A \gamma r^{\gamma - 1} \in L^\infty (- \delta_0, T + \delta_0; L^{ \frac{\gamma}{\gamma - 1} } (\widetilde{\Omega}) )$ and $\partial_t {\rm p}'(r) = A \gamma ( \gamma - 1 ) r^{\gamma - 2} \partial_t r \in L^1 (- \delta_0, T + \delta_0; L^{ \frac{\gamma}{\gamma - 1} } (\widetilde{\Omega}) )$, by Lemma \ref{mollifier-L-infty}, we get, as $\delta \rightarrow 0$,
\begin{equation*}
  \begin{aligned}
    \| I_2 \|_{ L^1 (0, T; L^{ \frac{\gamma}{\gamma-1} } ) } \leq C \| \dv {\bf v} \|_{ L^1 (0, T; L^\infty) } \| [ {\rm p}'(r_{t,x}) ]^\delta - {\rm p}'(r) \|_{ L^\infty (0, T; L^{ \frac{\gamma}{\gamma - 1} }) } \rightarrow 0 \,.
  \end{aligned}
\end{equation*}
For the term $I_3$, taking a similar way to \eqref{w-delta} and \eqref{I} at Lemma \ref{mollifier-L-infty} in Appendix \ref{Sec:appendix}, we have
\begin{align*}
    I_3= & - \int_{ \mathbb{R} } \int_{ \mathbb{R}^n } \Big ( [ {\rm p}' (r) ] (t, x) - [ {\rm p}'(r) ] (t - \delta s, x) \Big ) \dv ( {\bf v}(t, x) \\
    & \quad \quad \quad\quad - {\bf v} (t - \delta s, x - \delta y) ) \eta (y) \dd y \tilde{ \eta } (s) \dd s \\
    & - \int_{ \mathbb{R} } \int_{ \mathbb{R}^n } \Big ( [ {\rm p}' (r) ] (t - \delta s, x) - [ {\rm p}'(r) ] (t - \delta s, x - \delta y) \Big ) \dv ( {\bf v}(t, x) \\
    & \quad \quad \quad\quad - {\bf v} (t - \delta s, x - \delta y) ) \eta (y) \dd y \tilde{ \eta } (s) \dd s \\
    \leq & C \delta \int_0^T \int | \partial_t {\rm p}' (r) | | \dv {\bf v} | \dd x \dd t + C \delta \int_0^T \int | \nabla {\rm p}' (r) | | \dv {\bf v} | \dd x \dd t \,.
\end{align*}
Thanks to $r \geq r_1$ if $1 < \gamma \leq 2$, $r_1 \leq r \leq r_2$ if $\gamma > 2$, $\partial_t r \in L^2 (- \delta_0, T + \delta_0; L^{ \frac{2 \gamma}{\gamma + 1} } ( \widetilde{\Omega} ) \cap L^{ \frac{6}{5} } ( \widetilde{\Omega} ))$, $\nabla r \in L^2 (- \delta_0, T + \delta_0; L^{ \frac{2 \gamma}{\gamma + 1} } ( \widetilde{\Omega} ) \cap L^{ \frac{6}{5} } ( \widetilde{\Omega} ))$ and $\nabla {\bf v} \in L^2 (- \delta_0, T + \delta_0; L^{ \frac{2 \gamma}{\gamma - 1} } ( \widetilde{\Omega} ) \cup L^6 ( \widetilde{\Omega} ))$, and by H\"{o}lder's inequality, it follows that, as $\delta \rightarrow 0$,
\begin{equation*}
  \begin{aligned}
    \| I_3 \|_{ L^1 (0, T; L^{ \frac{\gamma}{\gamma-1} } ) } \rightarrow 0 \,.
  \end{aligned}
\end{equation*}
Therefore, we have, as $\delta \rightarrow 0$,
\begin{equation}\label{p'(r)-div-v-delta}
  \begin{aligned}
    {\rm p}' ( r_{t,x}^\delta ) \dv {\bf v}_{t,x}^\delta - ( {\rm p}' (r) \dv {\bf v} )_{t,x}^\delta \rightarrow 0 \ \ \textrm{ in } \; L^1 (0, T; L^{ \frac{\gamma}{\gamma-1} } ) \,.
  \end{aligned}
\end{equation}

Putting \eqref{P''-E1}, \eqref{P''E1-delta}, \eqref{dt-P'(r)-delta}, \eqref{v-grad-P'-delta} and \eqref{p'(r)-div-v-delta} together, it confirms that
\begin{equation}\label{P''E1(r-v-delta)}
  \begin{aligned}
    {\rm P}'' (r_{t,x}^\delta) E_1 (r_{t,x}^\delta, {\bf v}_{t,x}^\delta) \rightarrow {\rm P}'' (r) E_1 (r, {\bf v}) \ \ \textrm{ strongly in } \; L^1 (0, T; L^{ \frac{\gamma}{\gamma-1} } ) \,, \textrm{ as } \delta \rightarrow 0 \,.
  \end{aligned}
\end{equation}

Using the definition of $E_2 (r, {\bf v})$ in \eqref{E1-2}, by Lemma \ref{mollifier-derivative}, we express $\frac{1}{r_{t,x}^\delta} E_2 ( r_{t,x}^\delta, {\bf v}_{t,x}^\delta )$ as
\begin{align*}
    \frac{1}{r_{t,x}^\delta} E_2 (r_{t,x}^\delta, {\bf v}_{t,x}^\delta) = & \partial_t {\bf v}_{t,x}^\delta + {\bf v}_{t,x}^\delta \cdot \nabla {\bf v}_{t,x}^\delta + \frac{1}{r_{t,x}^\delta} \nabla p ( r_{t,x}^\delta ) - \frac{1}{r_{t,x}^\delta} \dv \mathbb{S} ( \nabla {\bf v}_{t,x}^\delta ) \\
    = & \Big ( \frac{1}{r} E_2 (r, {\bf v}) \Big )_{t,x}^\delta + P'' ( r_{t,x}^\delta ) \nabla r_{t,x}^\delta - ( P'' ( r ) \nabla r )_{t,x}^\delta \\
    & + {\bf v}_{t,x}^\delta \cdot \nabla {\bf v}_{t,x}^\delta - ( {\bf v} \cdot \nabla {\bf v} )_{t,x}^\delta - \frac{1}{r_{t,x}^\delta} \dv \mathbb{S} ( \nabla {\bf v}_{t,x}^\delta ) + \frac{1}{r} \dv \mathbb{S} ( \nabla {\bf v} ) \,.
\end{align*}
Taking a similar argument to \eqref{P''E1(r-v-delta)}, we have
\begin{equation}\label{1/rE2(r-v-delta)}
  \begin{aligned}
    \frac{1}{r_{t,x}^\delta} E_2 (r_{t,x}^\delta, {\bf v}_{t,x}^\delta) \rightarrow \frac{1}{r} E_2 (r, {\bf v}) \ \ \textrm{ strongly in } \, L^1 ( 0, T; L^{ \frac{2\gamma}{\gamma-1} } ) \,.
  \end{aligned}
\end{equation}

Recalling the expression $\Lambda ( \cdot )$ in \eqref{Lam(v)}, noticing that $\| \eta^\delta \|_{L^1 (\mathbb{R}^3)} = 1$ and $\| \tilde{\eta}^\delta \|_{L^1 (\mathbb{R})} = 1$, and by Lemmas \ref{convolution} and \ref{mollifier-derivative}, it yields
\begin{equation}\label{Lam(v-delta)}
  \begin{aligned}
    \Lambda ( {\bf v}_{t,x}^\delta ) \leq \Lambda ( {\bf v} ) \,.
  \end{aligned}
\end{equation}

By \eqref{grad-v-delta}, \eqref{div-v-delta}, \eqref{v-delta} and \eqref{r-delta-infty}, and note  the definition of ${\rm LS} (\rho, {\bf u}; \cdot, \cdot)$ in \eqref{relative-energy-ineq}, we have
\begin{equation*}
  \begin{aligned}
    & {\rm LS} (\rho, {\bf u}; r_{t,x}^\delta, {\bf v}_{t,x}^\delta) \rightarrow {\rm LS} (\rho, {\bf u}; r, {\bf v}) \, \textrm{ as } \delta \rightarrow 0 \,.
  \end{aligned}
\end{equation*}
In view of \eqref{v-delta}, \eqref{r-delta-infty}, \eqref{P''E1(r-v-delta)}, \eqref{1/rE2(r-v-delta)} and \eqref{Lam(v-delta)}, and note that  ${\rm RS} (\rho, {\bf u}; \cdot, \cdot)$ in \eqref{relative-energy-ineq}, we have
\begin{equation*}
  \begin{aligned}
    & {\rm RS} (\rho, {\bf u}; r_{t,x}^\delta, {\bf v}_{t,x}^\delta) \leq {\rm RS} (\rho, {\bf u}; r, {\bf v}) \, \textrm{ as } \delta \rightarrow 0 \,.
  \end{aligned}
\end{equation*}
So we proved it.

\end{proof}

\smallskip

\section{ Proof of Theorem \ref{thm-existence-dissipative-sol} }\label{Sec:thm-exist}

This section aims to present the proof of Theorem \ref{thm-existence-dissipative-sol} with focusing on  the periodic case. We will also point out
the difference between the Dirichlet boundary case and the periodic case.

Our proof is starting with the following modified Brenner model:
\begin{align}\label{approxi-system}
\left\{
 \begin{array}{ll}
    \partial_t \rho^\eps + \dv ( \rho^\eps {\bf u}^\eps ) = \eps \Delta \rho^\eps \,, \\
    \partial_t (\rho^\eps {\bf u}^\eps) + \dv ( \rho^\eps {\bf u}^\eps \otimes {\bf u}^\eps ) + \nabla {\rm p}(\rho^\eps) + \eps^a \nabla (\rho^\eps)^\beta = \dv \mathbb{S} ( \nabla {\bf u}^\eps ) + \eps \dv ( {\bf u}^\eps \otimes \nabla \rho^\eps ) \,,
 \end{array}\right.
\end{align}
where $\eps \in (0, 1]$ is a small parameter, $\beta > \max \{ 4, \gamma \}$, and $a$ is any positive constant. Here, we put the artificial pressure term $\eps^a \nabla (\rho^\eps)^\beta$ and the artificial diffusion term $\eps \Delta \rho^\eps$ at the same level, which differs from the approximation model in \cite{FNP-jmfm-2001}. We consider the approximate initial data
\begin{equation}\label{approxi-initial}
  \begin{aligned}
    ( \rho^\eps,  {\bf u}^\eps ) |_{t = 0} = ( \rho^\eps_0, {\bf u}^\eps_0 ) \,,
  \end{aligned}
\end{equation}
satisfying
\begin{align}\label{approxi-initial-condi}
    & \rho^\eps_0 \in C^3 ( \Omega ) \,, \;\; \int \rho^\epsilon_0 \dd x \geq C_1 >0 \,, \;\;  0 < \epsilon \leq \rho^\eps_0 \leq \epsilon^{-\frac{a}{2\beta}} \,, \;\; {\bf u}^\eps_0 \in C^3 ( \Omega ) \,, \nonumber\\
    & \rho^\eps_0 \rightarrow \rho_0 \textrm{ strongly in } L^{\gamma} \,, \;\; \sqrt{\rho^\eps_0} {\bf u}^\eps_0 \rightarrow \frac{{\bf m}_0}{\sqrt{\rho_0}} \textrm{ strongly in } L^2 \,, \textrm{ as } \eps \rightarrow0 \,,
\end{align}
where the positive constant $C_1 \leq C_{\rho_0}$ is independent of $\epsilon$, and $( \rho_0, {\bf m}_0 )$ satisfies \eqref{periodic-initial}. When we consider the Dirichlet boundary case, the condition $\int \rho^\epsilon_0 \dd x \geq C_1 >0$ in \eqref{approxi-initial-condi} can be removed and $( \rho_0, {\bf m}_0 )$ satisfies \eqref{Dirichlet-initial}. It needs to add the boundary conditions $\nabla \rho^\eps \cdot {\bf n} |_{\partial \Omega} = 0$ and ${\bf u}^\eps |_{\partial \Omega} = 0$.

For any fixed $\eps >0$ and any $T \in (0, \infty)$, by the Faedo-Galerkin approximation adopted by Feireisl et al. (\cite{FNP-jmfm-2001}, Proposition 2.1), the system \eqref{approxi-system}-\eqref{approxi-initial} has a global weak solution $( \rho^\eps, {\bf u}^\eps )$, which satisfies the energy differential inequality
\begin{align}\label{dt-rho-uu}
    & {\rm \frac{d}{dt}} \int \frac{1}{2} \rho^\eps |{\bf u}^\eps|^2 + {\rm P}(\rho^\eps) + \eps^a {\rm Q}(\rho^\eps) \dd x + \int \mathbb{S} (\nabla {\bf u}^\eps) : \nabla {\bf u}^\eps \dd x \nonumber\\
    & + \int \eps {\rm P}''(\rho^\eps) |\nabla \rho^\eps|^2 + \eps^{1+a} {\rm Q}''(\rho^\eps) |\nabla \rho^\eps|^2 \dd x \leq 0 \,,
\end{align}
for any $t \in [ 0, T ]$, where ${\rm P}(\rho^\eps) = \frac{A}{\gamma - 1} (\rho^\eps)^\gamma$ and ${\rm Q}(\rho^\eps) = \frac{1}{\beta - 1} (\rho^\eps)^\beta$.

By \eqref{approxi-initial-condi} and Lemma \ref{Korn-type-ineq}, it follows from \eqref{dt-rho-uu} that
\begin{align}
    & \sup_{t \in [0, T]} \| \sqrt{ \rho^\eps } {\bf u}^\eps \|_{L^2} \leq C \,, \label{esti-rhou-L2}\\
    & \sup_{t \in [0,T]} \| \rho^\eps \|_{L^\gamma} \leq C  \,, \label{esti-rho-Lgamma}\\
    & \int_0^T \| {\bf u}^\eps \|_{H^1} \dd t \leq C \,, \label{esti-u-H1}\\
    & \eps \int_0^T \int | \rho^\eps |^{\gamma-2} |\nabla \rho^\eps|^2 \dd x \dd t \leq C \,. \label{esti-eps-rho-L2}
\end{align}
Here, the estimate \eqref{esti-u-H1} requires $\gamma \geq \frac{6}{5}$ from Lemma \ref{Korn-type-ineq}. (For the Dirichlet boundary case, the estimate \eqref{esti-u-H1} is hold for $\gamma > 1$ by Poincar\'{e}'s inequality.)

The conclusion (\cite{FNP-jmfm-2001}, Proposition 2.1) tells us that
\begin{equation*}
  \begin{aligned}
    \partial_t \rho^\eps + \dv ( \rho^\eps {\bf u}^\eps ) = \eps \Delta \rho^\eps \ \ \textrm{ a.e. } x \in (0,T) \times \Omega \,.
  \end{aligned}
\end{equation*}
Multiplying $B' (\rho^\eps)$ on the both sides of above equation, we arrive at
\begin{align}\label{renorm-conti-b}
    & \partial_t B ( \rho^\eps ) + \dv ( B(\rho^\eps) {\bf u}^\eps ) + ( B'(\rho^\eps) \rho^\eps - B(\rho^\eps) ) \dv {\bf u}^\eps \nonumber\\
    = & \eps \dv ( B'(\rho^\eps) \nabla \rho^\eps ) - \eps B''(\rho^\eps) | \nabla \rho^\eps |^2  \,,
\end{align}
where $B \in C ( [0, \infty) ) \cap C^2 ( (0, \infty) )$ with $B'(z) = 0$ for large $z \in \mathbb{R}^+$.

On the one hand, taking $B(z) = z \ln z$ for $z \in [0, 1]$ in \eqref{renorm-conti-b}, and integrating the result over $(0,T) \times \{ x \in \Omega : 0 < \rho^\eps \leq 1 \}$, it follows from \eqref{esti-rho-Lgamma} and \eqref{esti-u-H1} that
\begin{align*}
    & \eps \int_0^T \int_{ \{ x : 0 < \rho^\eps \leq 1 \} } ( \rho^\eps )^{-1} | \nabla \rho^\eps |^2 \dd x \dd t \\
    = & - \int_{ \{ x : \rho^\eps \leq 1 \} } \rho^\eps \ln \rho^\eps \dd x \big |_0^T - \int_0^T \int_{ \{ x : \rho^\eps \leq 1 \} } \rho^\eps \dv {\bf u}^\eps \dd x \dd t \leq C \,.
\end{align*}
On the other hand, by \eqref{esti-eps-rho-L2}, it infers that
\begin{align*}
    \eps \int_0^T \int_{ \{ x : \rho^\eps \geq 1 \} } ( \rho^\eps )^{-1} | \nabla \rho^\eps |^2 \dd x \dd t \leq \eps \int_0^T \int ( \rho^\eps )^{ \gamma - 1 } \Big| \frac{\nabla \rho^\eps}{ (\rho^\eps)^\frac{1}{2} } \Big|^2 \dd x \dd t \leq C \,.
\end{align*}
Then, it implies that
\begin{equation}\label{rho-1/2-grad-rho}
  \begin{aligned}
    \eps^{ \frac12 } (\rho^\eps)^{-\frac12} \nabla \rho^\eps \in L^2 (0,T; L^2) \,.
  \end{aligned}
\end{equation}

Similar to \eqref{rho-1/2-grad-rho}, if we take $B(z) = - 4 \sqrt{z}$ for $z \in [0, 1]$, then we have
\begin{equation}\label{rho-3/4-grad-rho}
  \begin{aligned}
    \eps^{ \frac12 } (\rho^\eps)^{-\frac34} \nabla \rho^\eps \in L^2 (0,T; L^2) \,.
  \end{aligned}
\end{equation}

By the estimates \eqref{esti-rho-Lgamma} and \eqref{rho-1/2-grad-rho}, and  H\"{o}lder's inequality, one has
\begin{equation}\label{grad-rho}
  \begin{aligned}
    \eps^{ \frac{1}{2} } \nabla \rho^\eps \in L^2 (0, T; L^{ \frac{2\gamma}{\gamma + 1} }) \,.
  \end{aligned}
\end{equation}

Using the energy estimates \eqref{esti-rho-Lgamma} and \eqref{esti-u-H1}, up to a subsequence $(\rho^\eps, {\bf u}^\eps)$ without relabeled, there exists a weak limit $(\rho, {\bf u})$ such that
\begin{align}
    & \rho^\eps \rightharpoonup \rho \ \ \textrm{ weakly-} \star \textrm{ in } \; L^\infty (0, T; L^\gamma) \,, \label{rho-limit-1} \\
    & {\bf u}^\eps \rightharpoonup {\bf u} \ \ \textrm{ weakly in } \; L^2 (0, T; H^1) \,. \label{u-limit}
\end{align}

Notice that $\sqrt{\rho^\eps} \in L^\infty (0, T; L^{2\gamma})$ from \eqref{esti-rho-Lgamma}, there exists a function $\tilde{\rho}$ such that
\begin{equation}\label{sqrt-rho-L2gamma}
  \begin{aligned}
    \sqrt{\rho^\eps} \rightharpoonup \sqrt{ \tilde{\rho} } \ \ \textrm{ weakly-} \star \textrm{ in } \; L^\infty (0, T; L^{2\gamma}) \,.
  \end{aligned}
\end{equation}
Taking $B(\rho^\eps) = \sqrt{\rho^\eps}$ in \eqref{renorm-conti-b}, and by the estimates \eqref{esti-rhou-L2}, \eqref{esti-rho-Lgamma}, \eqref{esti-u-H1}, \eqref{rho-1/2-grad-rho} and \eqref{rho-3/4-grad-rho}, then we have
\begin{align}\label{pt-sqrt-rho}
    \partial_t (\sqrt{\rho^\eps}) = & - \dv (\sqrt{\rho^\eps} {\bf u}^\eps) + \frac12 \sqrt{\rho^\eps} \dv {\bf u}^\eps + \eps \dv \Big ( \frac{1}{ 2 \sqrt{\rho^\eps} } \nabla \rho^\eps \Big ) + \frac14 \eps (\rho^\eps)^{-\frac32} | \nabla \rho^\eps |^2 \nonumber\\
    \in & L^\infty (0, T; W^{-1,2}) + L^2 (0, T; L^{ \frac{2\gamma}{\gamma+1} }) + L^2 (0, T; W^{-1,2}) +L^2(0, T; L^2) \nonumber\\
    \subset & L^2 (0, T; W^{-1, \frac{2\gamma}{\gamma+1}}) \,.
\end{align}
With the help of Lemma \ref{product-weak-converge}, we have
\begin{equation}\label{rho-limit-2}
  \begin{aligned}
    \rho^\eps = \sqrt{ \rho^\eps } \sqrt{ \rho^\eps } \rightarrow \tilde{ \rho } \ \ \textrm{ in } \; \mathscr{D}' ( (0, T) \times \Omega ) \,.
  \end{aligned}
\end{equation}
Combining \eqref{rho-limit-1} and \eqref{rho-limit-2}, the uniqueness of limit implies, for a.e. $(t, x) \in (0, T) \times \Omega$,
\begin{equation}\label{unique-rho}
  \begin{aligned}
    \rho = \tilde{\rho} \,.
  \end{aligned}
\end{equation}

By \eqref{sqrt-rho-L2gamma}, \eqref{pt-sqrt-rho} and \eqref{unique-rho}, and using Lemma \ref{conti-time-weak} gives
\begin{equation*}
    \sqrt{\rho^\eps} \rightarrow \sqrt{\rho} \ \ \textrm{ in } \; C_w ( [0, T]; L^{2\gamma} ) \,.
\end{equation*}
Since $2\gamma > 2 > \frac{6}{5}$, by the interpolation relation $L^{2\gamma} \hookrightarrow\hookrightarrow H^{-1}$, we know
\begin{equation}\label{sqrt-rho-H-1}
  \begin{aligned}
    \sqrt{\rho^\eps} \rightarrow \sqrt{\rho} \ \  \textrm{ in } \; C( [0, T]; H^{-1} ) \,.
  \end{aligned}
\end{equation}

It follows from \eqref{u-limit} and \eqref{sqrt-rho-H-1} that
\begin{equation*}
    \sqrt{\rho^\eps} {\bf u}^\eps \rightarrow \sqrt{ \rho } {\bf u} \ \ \textrm{ in } \; \mathscr{D}' ( (0, T) \times \Omega ) \,.
\end{equation*}
Since $\sqrt{\rho^\eps} {\bf u}^\eps \in L^\infty (0, T; L^2)$ from \eqref{esti-rhou-L2}, one has
\begin{equation}\label{sqrt-rho-u}
  \begin{aligned}
    \sqrt{\rho^\eps} {\bf u}^\eps \rightharpoonup \sqrt{ \rho } {\bf u} \ \ \textrm{ weakly-} \star \textrm{ in } \; L^\infty (0, T; L^2) \,.
  \end{aligned}
\end{equation}

By \eqref{rho-limit-1}, \eqref{pt-sqrt-rho} and \eqref{sqrt-rho-u}, in view of Lemma \ref{product-weak-converge}, we have
\begin{equation*}
  \begin{aligned}
    \rho^\eps {\bf u}^\eps =\sqrt{ \rho^\eps } \sqrt{ \rho^\eps } { \bf u}^\eps \rightarrow \rho {\bf u} \ \ \textrm{ in } \; \mathscr{D}' ( (0, T) \times \Omega ) \,.
  \end{aligned}
\end{equation*}
Since $\rho^\eps {\bf u}^\eps \in L^\infty (0, T; L^{ \frac{2\gamma}{\gamma + 1} })$ from \eqref{esti-rhou-L2} and \eqref{esti-rho-Lgamma}, it implies
\begin{equation}\label{rho-u}
  \begin{aligned}
    \rho^\eps {\bf u}^\eps \rightharpoonup \rho {\bf u} \ \ \textrm{ weakly-} \star \textrm{ in } \; L^\infty (0, T; L^{ \frac{2\gamma}{\gamma + 1} }) \,.
  \end{aligned}
\end{equation}

From the discussion in Section \ref{Sec:regular}, we can choose smooth functions $r > 0$ and ${\bf v}$ satisfying the regularities \eqref{ds-regular1} and \eqref{ds-regular2}. Taking the inner production with $\eqref{approxi-system}_2$ by $- {\bf v}$, multiplying $\eqref{approxi-system}_1$ by $\frac{1}{2} |{\bf v}|^2 - {\rm P}'(r)$, integrating the result over $\Omega$ with respect to $x$, and by integration by parts, it deduces that
\begin{align}\label{dt-v-v2-P'(r)-Q'(r)}
    & {\rm \frac{d}{dt}} \int - \rho^\eps {\bf u}^\eps \cdot {\bf v} + \frac{1}{2} \rho^\eps |{\bf v}|^2 - {\rm P}'(r) \rho^\eps \dd x \nonumber\\
    = & \int \rho^\eps \partial_t {\bf v} \cdot ( {\bf v} - {\bf u}^\eps ) \dd x + \int \rho^\eps ( {\bf u}^\eps \cdot \nabla ) {\bf v} \cdot ( {\bf v} - {\bf u}^\eps ) \dd x - \int {\rm p}(\rho^\eps) \dv {\bf v} \dd x \nonumber\\
    & - \eps^a \int (\rho^\eps)^\beta \dv {\bf v} \dd x + \int \mathbb{S} (\nabla {\bf u}^\eps) : \nabla {\bf v} \dd x + \eps \int (\nabla \rho^\eps \cdot \nabla ) {\bf v} \cdot ( {\bf u}^\eps - {\bf v} ) \dd x \nonumber\\
    & - \int \rho^\eps \partial_t {\rm P}'(r) \dd x - \int \rho^\eps {\bf u}^\eps \cdot \nabla {\rm P}'(r) \dd x + \eps \int \nabla \rho^\eps \cdot \nabla {\rm P}'(r) \dd x \,.
\end{align}
Note that \eqref{pressure}, we rewrite it as
\begin{align}\label{dt-p(r)-Q(r)}
    {\rm \frac{d}{dt}} \int {\rm P}'(r) r - {\rm P}(r) \dd x = \int r \partial_t {\rm P}'(r) \dd x + \int r {\bf v} \cdot \nabla {\rm P}'(r) + {\rm p}(r) \dv {\bf v} \dd x \,.
\end{align}

Adding up \eqref{dt-rho-uu}, \eqref{dt-v-v2-P'(r)-Q'(r)} and \eqref{dt-p(r)-Q(r)}, and by means of the definitions of $E_1(r, {\bf v})$ and $E_2(r, {\bf v})$ in \eqref{E1-2}, we arrive at
\begin{equation*}
  \begin{aligned}
    & {\rm \frac{d}{dt}} \int \frac{1}{2} \rho^\eps |{\bf u}^\eps - {\bf v}|^2 + {\rm P}(\rho^\eps) - {\rm P}'(r) ( \rho^\eps - r ) - {\rm P}(r) + \eps^a {\rm Q}(\rho^\eps) \dd x \\
    & + \int \mathbb{S} ( \nabla {\bf u}^\eps - \nabla {\bf v} ) : \nabla ( {\bf u}^\eps - {\bf v} ) \dd x + \int \eps {\rm P}''(\rho^\eps) |\nabla \rho^\eps|^2 + \eps^{1+a} {\rm Q}''(\rho^\eps) |\nabla \rho^\eps|^2 \dd x \\
    \leq & - \int \rho ( {\bf u}^\eps - {\bf v} ) \cdot \mathbb{D} ( {\bf v} ) \cdot ( {\bf u}^\eps - {\bf v} ) \dd x - \int ( {\rm p}(\rho^\eps) - {\rm p}'(r) ( \rho^\eps - r ) - {\rm p}(r) ) \dv {\bf v} \dd x \\
    & - \eps^a \int (\rho^\eps)^\beta \dv {\bf v} \dd x + \int ( r - \rho^\eps ) {\rm P}''(r) E_1 ( r, {\bf v} ) \dd x + \int \frac{\rho^\eps}{r} E_2 ( r, {\bf v} ) \cdot ( {\bf v} - {\bf u}^\eps ) \dd x \\
    & + \int \frac{\rho^\eps - r}{r} \dv \mathbb{S} (\nabla {\bf v}) \cdot ( {\bf v} - {\bf u}^\eps ) \dd x + \eps \int (\nabla \rho^\eps \cdot \nabla ) {\bf v} \cdot ( {\bf u}^\eps - {\bf v} ) \dd x + \eps \int \nabla \rho^\eps \cdot \nabla {\rm P}'(r) \dd x \,.
  \end{aligned}
\end{equation*}
Using the same way as \eqref{r-S-u-v-torus} to deal with the third term from bottom on the right-hand side of the above equality, then applying Gr\"{o}nwall's inequality to the result gives
\begin{align}\label{relative-energy-ineq-eps}
    & \int \frac{1}{2} \rho^\eps |{\bf u}^\eps - {\bf v}|^2 + {\rm P}(\rho^\eps) - {\rm P}'(r) ( \rho^\eps - r ) - {\rm P}(r) + \eps^a {\rm Q}(\rho^\eps) \dd x \nonumber\\
    & + \frac{1}{2} \int_0^t \int \mathbb{S} ( \nabla {\bf u}^\eps - \nabla {\bf v} ) : \nabla ( {\bf u}^\eps - {\bf v} ) \dd x \dd s \leq \mathscr{R}_0 + \sum_{i=1}^4 \mathscr{R}_i \,,
\end{align}
where
\begin{align*}
    & \mathscr{R}_0 = \exp (\int_0^t C_0 \Lambda ( {\bf v} ) \dd s ) \int\big( \frac{1}{2} \rho^\eps_0 | {\bf u}^\eps_0 - {\bf v}_0 |^2 + {\rm P}(\rho^\eps_0) - {\rm P}'(r_0) (\rho^\eps_0 - r_0) - {\rm P}(r_0) + \eps^a {\rm Q}(\rho^\eps_0) \big)\dd x \,, \\
    & \mathscr{R}_1 = \int_0^t \int \exp \Big ( \int_s^t C_0 \Lambda ( {\bf v} ) \dd \tau \Big ) ( r - \rho^\eps ) {\rm P}''(r) E_1 (r, {\bf v}) \dd x \dd s \,, \\
    & \mathscr{R}_2 = \int_0^t \int \exp \Big ( \int_s^t C_0 \Lambda ( {\bf v} ) \dd \tau \Big ) \frac{\rho^\eps}{r} E_2 (r, {\bf v}) \cdot ( {\bf v} - {\bf u}^\eps ) \dd x \dd s \,, \\
    & \mathscr{R}_3 = \eps \int_0^t \int \exp \Big ( \int_s^t C_0 \Lambda ( {\bf v} ) \dd \tau \Big ) (\nabla \rho^\eps \cdot \nabla ) {\bf v} \cdot ( {\bf u}^\eps - {\bf v} ) \dd x \dd s \,, \\
    & \mathscr{R}_4 = \eps \int_0^t \int \exp \Big ( \int_s^t C_0 \Lambda ( {\bf v} ) \dd \tau \Big ) \nabla \rho^\eps \cdot \nabla {\rm P}'(r) \dd x \dd s \,.
\end{align*}

The next step is to recover a dissipative solution by passing to the limits in  \eqref{relative-energy-ineq-eps} as $\varepsilon$ tends to zero.
We first deal with the left-hand side of \eqref{relative-energy-ineq-eps}. By the weak convergences \eqref{sqrt-rho-u} and \eqref{u-limit}, and in view of the low semi-continuous of $L^2$-norm, we get, as $\eps \rightarrow 0$,
\begin{align*}
    & \int \frac{1}{2} \rho^\eps | {\bf u}^\eps |^2 \dd x = \int \frac{1}{2} | \sqrt{\rho^\eps} {\bf u}^\eps |^2 \dd x \geq \int \frac{1}{2} \rho | {\bf u} |^2 \dd x \,,\\
    & \int_0^t \int \mathbb{S} ( \nabla {\bf u}^\eps ) : \nabla ( {\bf u}^\eps ) \dd x \dd s \geq \int_0^t \int \mathbb{S} ( \nabla {\bf u} ) : \nabla ( {\bf u} ) \dd x \dd s \,.
\end{align*}
It follows from \eqref{rho-u}, \eqref{rho-limit-1} and \eqref{u-limit} that, as $\eps \rightarrow 0$,
\begin{align*}
    & - \int \rho^\eps {\bf u}^\eps \cdot {\bf v} \dd x \rightarrow - \int \rho {\bf u} \cdot {\bf v} \dd x \,, \quad \int \frac{1}{2} \rho^\eps | {\bf v} |^2 \dd x \rightarrow \int \frac{1}{2} \rho | {\bf v} |^2 \dd x \,, \\
    & - \int_0^t \int \mathbb{S} ( \nabla {\bf u}^\eps ) : \nabla {\bf v} + \mathbb{S} ( \nabla {\bf v} ) : \nabla {\bf u}^\eps \dd x \dd s \rightarrow - \int_0^t \int \mathbb{S} ( \nabla {\bf u} ) : \nabla {\bf v} + \mathbb{S} ( \nabla {\bf v} ) : \nabla {\bf u} \dd x \dd s \,.
\end{align*}
By the convexity of ${\rm P}(\cdot)$ and \eqref{rho-limit-1}, one has, as $\eps \rightarrow 0$,
\begin{align*}
    \int {\rm P}(\rho^\eps) - {\rm P}'(r) ( \rho^\eps - r ) - {\rm P}(r) \dd x \geq \int {\rm P}(\rho) - {\rm P}'(r) ( \rho - r ) - {\rm P}(r) \dd x \,.
\end{align*}
Then, as $\eps \rightarrow 0$,
\begin{align}\label{left-side}
    & \int \frac{1}{2} \rho^\eps |{\bf u}^\eps - {\bf v}|^2 + {\rm P}(\rho^\eps) - {\rm P}'(r) ( \rho^\eps - r ) - {\rm P}(r) + \eps^a {\rm Q}(\rho^\eps) \dd x \nonumber\\
    & + \frac{1}{2} \int_0^t \int \mathbb{S} ( \nabla {\bf u}^\eps - \nabla {\bf v} ) : \nabla ( {\bf u}^\eps - {\bf v} ) \dd x \dd s \nonumber\\
    \geq & \int \frac{1}{2} \rho | {\bf u} - {\bf v} |^2 + {\rm P}(\rho) - {\rm P}'(r) ( \rho - r ) - {\rm P}(r) \dd x \nonumber\\
    & + \frac{1}{2} \int_0^t \int \mathbb{S} ( \nabla {\bf u} - \nabla {\bf v} ) : \nabla ( {\bf u} - {\bf v} ) \dd x \dd s \,.
\end{align}

Next, we tackle with the reminder terms $\mathscr{R}_i \; (i = 0,1,2,3,4)$ on the right-hand side of \eqref{relative-energy-ineq-eps}. Before this, recalling the definition $\Lambda ( {\bf v} )$ in \eqref{Lam(v)} and the regularities $\mathbb{D} ( {\bf v} ) \in L^1 (0, T; L^\infty)$ and $\dv \mathbb{S} ( \nabla {\bf v} ) \in L^2 (0, T; L^{ \frac{6 \gamma}{5 \gamma - 3} }) \cap L^1 (0, T; L^{ \frac{2 \gamma}{\gamma - 1} })$ in \eqref{ds-regular1}, we know that
\begin{equation*}
  \begin{aligned}
    \exp \Big ( \int_s^t C_0 \Lambda ( {\bf v} ) \dd \tau \Big ) \leq \exp \Big ( \int_0^t C_0 \Lambda ( {\bf v} ) \dd \tau \Big ) \leq C \,.
  \end{aligned}
\end{equation*}
Now, we begin to deal with the terms $\mathscr{R}_i$. By the initial conditions \eqref{approxi-initial-condi}, we have, as $\eps \rightarrow 0$,
\begin{equation}\label{R0}
  \begin{aligned}
    \mathscr{R}_0 \rightarrow \exp \Big ( \int_0^t C_0 \Lambda ( {\bf v} ) \dd s \Big ) \int \frac{1}{2} \rho_0 \Big | \frac{{\bf m}_0}{\rho_0} - {\bf v}_0 \Big |^2 + {\rm P}(\rho_0) - {\rm P}'(r_0) (\rho_0 - r_0) - {\rm P}(r_0) \dd x \,.
  \end{aligned}
\end{equation}
For the terms $\mathscr{R}_1$ and $\mathscr{R}_2$, it follows from \eqref{rho-limit-1} and \eqref{rho-u} that, as $\eps \rightarrow 0$,
\begin{align}
    & \mathscr{R}_1 \rightarrow \int_0^t \int \exp \Big ( \int_s^t C_0 \Lambda ( {\bf v} ) \dd \tau \Big ) ( r - \rho ) {\rm P}''(r) E_1 (r, {\bf v}) \dd x \dd s \,, \label{R1} \\
    & \mathscr{R}_2 \rightarrow \int_0^t \int \exp \Big ( \int_s^t C_0 \Lambda ( {\bf v} ) \dd \tau \Big ) \frac{\rho}{r} E_2 (r, {\bf v}) \cdot ( {\bf v} - {\bf u} ) \dd x \dd s \,. \label{R2}
\end{align}
We turn to the term $\mathscr{R}_3$. By \eqref{rho-1/2-grad-rho}, \eqref{grad-rho} and \eqref{sqrt-rho-u}, it implies that
\begin{align}
    & \eps ( \sqrt{\rho^\eps} )^{-1} \nabla \rho^\eps \cdot( \sqrt{\rho^\eps} {\bf u}^\eps ) \rightarrow 0 \ \ \textrm{ in } \; L^2 (0, T; L^1) \,, \label{eps-grad-rho-u} \\
    & \eps \nabla \rho^\eps \rightarrow 0 \ \ \textrm{ in } \; L^2 (0, T; L^{\frac{2\gamma}{\gamma + 1}}) \label{eps-grad-rho} \,.
\end{align}
In view of \eqref{eps-grad-rho-u} and \eqref{eps-grad-rho}, together with the regularities ${\bf v} \in C ([0, T]; L^{ \frac{2 \gamma}{\gamma - 1} })$ in \eqref{ds-regular1} and $\nabla {\bf v} \in L^2 (0, T; L^\infty)$ in \eqref{ds-regular2}, we have, as $\eps \rightarrow 0$,
\begin{align}\label{R3}
    \mathscr{R}_3 = & \eps \int_0^t \int \exp \Big ( \int_s^t C_0 \Lambda ( {\bf v} ) \dd \tau \Big ) ( \sqrt{\rho^\eps} )^{-1} \nabla \rho^\eps \cdot \nabla {\bf v} \cdot \sqrt{\rho^\eps} {\bf u}^\eps \dd x \dd s \nonumber\\
    & - \eps \int_0^t \int \exp \Big ( \int_s^t C_0 \Lambda ( {\bf v} ) \dd \tau \Big ) (\nabla \rho^\eps \cdot \nabla ) {\bf v} \cdot {\bf v} ] \dd x \dd s \rightarrow 0 \,.
\end{align}
Finally, for the terms $\mathscr{R}_4$, noticing that $r \geq r_1$ if $1 < \gamma \leq 2$, $r_1 \leq r \leq r_2$ if $\gamma > 2$ and $\nabla r \in L^\infty (0, T; L^{ \frac{2 \gamma}{\gamma - 1} })$ in \eqref{ds-regular1} and \eqref{ds-regular2}, by H\"{o}lder's inequality, it follows from \eqref{eps-grad-rho} that, as $\eps \rightarrow 0$,
\begin{equation}\label{R4}
  \begin{aligned}
    & \mathscr{R}_4 \leq C \| \eps \nabla \rho^\eps \|_{ L^2 (0, T; L^{\frac{2\gamma}{\gamma + 1}}) } \| \nabla r \|_{ L^2 (0, T; L^{ \frac{2 \gamma}{\gamma - 1} }) } \rightarrow 0 \,.
  \end{aligned}
\end{equation}

Thus, we complete the proof of Theorem \ref{thm-existence-dissipative-sol}.

\section{Proof of Theorem \ref{thm-weak-dissipative}}\label{Sec:ws-ds}
The goal of this section is to show  that the weak solution of the compressible isentropic Navier-Stokes equations is also a dissipative solution in the sense of Definition \ref{def-dissipative-sol}.

Let $(\rho, {\bf u})$ be the weak solution of the problem \eqref{com-NS}-\eqref{boundary1} or \eqref{com-NS}, \eqref{initial} and \eqref{boundary2}. Let $ \chi \in C_c^\infty ( (0,T) )$ and $\phi_m \in C_c^\infty ( \Omega )$. We can also let $(r, {\bf v})$ be the smooth functions satisfying \eqref{ds-regular1} and \eqref{ds-regular2} by the arguments in Section \ref{Sec:regular}. Taking the test function $\Phi = \chi \phi_m {\bf v}$ in \eqref{weak-mome-D} of Definition \ref{def-weak-sol} gives
\begin{align}\label{ws-rho-uv}
    0 = & \int_0^T \int \rho {\bf u} \cdot \partial_t ( \chi \phi_m {\bf v} ) \dd x \dd t + \int_0^T \int \rho {\bf u} \otimes {\bf u} : \nabla ( \chi \phi_m {\bf v} ) \dd x \dd t \nonumber\\
    & + \int_0^T \int {\rm p}(\rho) \dv ( \chi \phi_m {\bf v} ) \dd x \dd t - \int_0^T \int \mathbb{S} (\nabla {\bf u} ) : \nabla ( \chi \phi_m {\bf v} ) \dd x \dd t \nonumber\\
    = & \int_0^T \partial_t \chi \int \phi_m \rho {\bf u} \cdot {\bf v} \dd x \dd t + \int_0^T \chi \int \phi_m \rho {\bf u} \cdot \partial_t {\bf v} \dd x \dd t \nonumber\\
    & + \int_0^T \chi \int \phi_m \rho {\bf u} \cdot \nabla {\bf v} \cdot {\bf u} \dd x \dd t + \int_0^T \chi \int \phi_m {\rm p}(\rho) \dv {\bf v} \dd x \dd t \nonumber\\
    & - \int_0^T \chi \int \phi_m \mathbb{S} (\nabla {\bf u} ) : \nabla {\bf v} \dd x \dd t + \int_0^T \chi \int \rho {\bf u} \otimes {\bf u} : ( \nabla \phi_m \otimes {\bf v} ) \dd x \dd t \nonumber\\
    & + \int_0^T \chi \int {\rm p}(\rho) \nabla \phi_m \cdot {\bf v} \dd x \dd t - \int_0^T \chi \int \mathbb{S} (\nabla {\bf u} ) : ( \nabla \phi_m \otimes {\bf v} ) \dd x \dd t \,.
\end{align}

Choosing $\Psi = \frac{1}{2} |{\bf v}|^2 \chi \phi_m$ and $\Psi = {\rm P}'(r) \chi \phi_m$ in \eqref{weak-conti-D} of Definition \ref{def-weak-sol}, respectively, one sees that
\begin{align}\label{ws-rho-vv}
    0 = & \int_0^T \int \rho \partial_t \Big ( \frac12 | {\bf v} |^2 \chi \phi_m \Big ) \dd x \dd t + \int_0^T \int \rho {\bf u} \cdot \nabla \Big ( \frac12 | {\bf v} |^2 \chi \phi_m \Big ) \dd x \dd t \nonumber\\
    = & \int_0^T \partial_t \chi \int \phi_m \rho \Big ( \frac12 | {\bf v} |^2 \Big ) \dd x \dd t + \int_0^T \chi \int \phi_m \rho {\bf v} \cdot \partial_t {\bf v} \dd x \dd t \nonumber\\
    & + \int_0^T \chi \int \phi_m \rho {\bf u} \cdot \nabla {\bf v} \cdot {\bf v} \dd x \dd t + \int_0^T \chi \int \frac12 | {\bf v} |^2 \rho {\bf u} \cdot \nabla \phi_m \dd x \dd t \,,
\end{align}
and
\begin{align}\label{ws-rho-P'(r)}
    0 = & \int_0^T \int \rho \partial_t ( {\rm P}'(r) \chi \phi_m ) \dd x \dd t + \int_0^T \int \rho {\bf u} \cdot \nabla ( {\rm P}'(r) \chi \phi_m ) \dd x \dd t \nonumber\\
    = & \int_0^T \partial_t \chi \int \phi_m \rho {\rm P}'(r) \dd x \dd t + \int_0^T \chi \int \phi_m \rho \partial_t {\rm P}'(r) \dd x \dd t \nonumber\\
    & + \int_0^T \chi \int \phi_m \rho {\bf u} \cdot \nabla {\rm P}'(r) \dd x \dd t + \int_0^T \chi \int \rho {\rm P}'(r) {\bf u} \cdot \nabla \phi_m \dd x \dd t \,.
\end{align}

Noticing that the relations ${\rm P}'(r) r - {\rm P}(r) = {\rm p}(r)$ and ${\rm P}''(r) r = {\rm p}'(r)$, we have
\begin{align}\label{ws-P'(r)r-P(r)}
    0 = & \int_0^T {\rm \frac{d}{dt}} \int [ {\rm P}'(r) r - {\rm P}(r) ] \chi \phi_m \dd x \dd t \nonumber\\
    = & \int_0^T {\rm \frac{d}{dt}} \int {\rm p}(r) \chi \phi_m \dd x \dd t + \int_0^T \int \dv ( {\rm p}(r) {\bf v} \chi \phi_m ) \dd x \dd t \nonumber\\
    = & \int_0^T \partial_t \chi \int \phi_m [ {\rm P}'(r) r - {\rm P}(r) ] \dd x \dd t + \int_0^T \chi \int \phi_m r \partial_t {\rm P}'(r) \dd x \dd t \nonumber\\
    & + \int_0^T \chi \int \phi_m r {\bf v} \cdot \nabla {\rm P}'(r) \dd x \dd t + \int_0^T \chi \int \phi_m {\rm p}(r) \dv {\bf v} \dd x \dd t \nonumber\\
    & + \int_0^T \chi \int {\rm p}(r) {\bf v} \cdot \nabla \phi_m \dd x \dd t \,.
\end{align}

Collecting \eqref{ws-rho-uv}, \eqref{ws-rho-vv}, \eqref{ws-rho-P'(r)} and \eqref{ws-P'(r)r-P(r)} together, we obtain that
\begin{align}\label{ws-relative-energy1}
    & - \int_0^T \partial_t \chi \int \phi_m \Big [ - \rho {\bf u} \cdot {\bf v} + \frac{1}{2} | {\bf v} |^2 - {\rm P}'(r) \rho + {\rm P}'(r) r - {\rm P}(r) \Big ] \dd x \dd t \nonumber\\
    = & - \int_0^T \chi \int \phi_m \rho ( {\bf u} - {\bf v} ) \cdot \nabla {\bf v} \cdot ( {\bf u} - {\bf v} ) \dd x \dd t \nonumber\\
    & + \int_0^T \chi \int \phi_m \rho ( \partial_t {\bf v} + {\bf v} \cdot \nabla {\bf v} ) \cdot ( {\bf v} - {\bf u} ) \dd x \dd t \nonumber\\
    & + \int_0^T \chi \int \phi_m ( r - \rho ) \partial_t {\rm P}'(r) \dd x \dd t + \int_0^T \chi \int \phi_m ( r {\bf v} - \rho {\bf u} ) \cdot \nabla {\rm P}'(r) \dd x \dd t \nonumber\\
    & - \int_0^T \chi \int \phi_m ( {\rm p}(\rho) - {\rm p}(r) ) \dv {\bf v} \dd x \dd t + \int_0^T \chi \int \phi_m \mathbb{S} (\nabla {\bf u}) : \nabla {\bf v} \dd x \dd t \nonumber\\
    & + \mathcal{R}_{\chi \phi_m} \,,
\end{align}
where
\begin{align*}
    \mathcal{R}_{\chi \phi_m} = & \underbrace{ - \int_0^T \partial_t \chi \int ( 1 - \phi_m ) \Big [ \frac{1}{2} \rho | {\bf u} |^2 + {\rm P}(\rho) \Big ] \dd x \dd t }_{ \mathscr{X}_1 } \\
    & \underbrace{ - \int_0^T \chi \int \rho {\bf u} \otimes {\bf u} : ( \nabla \phi_m \otimes {\bf v} ) \dd x \dd t - \int_0^T \chi \int {\rm p}(\rho) {\bf v} \cdot \nabla \phi_m \dd x \dd t }_{ \mathscr{X}_2 } \\
    & + \underbrace{ \int_0^T \chi \int \mathbb{S} (\nabla {\bf u} ) : ( \nabla \phi_m \otimes {\bf v} ) \dd x \dd t }_{ \mathscr{X}_3 } + \underbrace{ \int_0^T \chi \int \frac12 | {\bf v} |^2 \rho {\bf u} \cdot \nabla \phi_m \dd x \dd t }_{ \mathscr{X}_4 } \\
    & \underbrace{ - \int_0^T \chi \int \rho {\rm P}'(r) {\bf u} \cdot \nabla \phi_m \dd x \dd t }_{ \mathscr{X}_5 } + \underbrace{ \int_0^T \chi \int {\rm p}(r) {\bf v} \cdot \nabla \phi_m \dd x \dd t }_{ \mathscr{X}_6 } \,.
\end{align*}

We endow a sequence $\phi_m \in C_c^\infty (\Omega)$ with
\begin{equation*}
  \begin{aligned}
    & 0 \leq \phi_m \leq 1 \,, \;\; \phi_m = 1 \textrm{ for } x \in \Omega \,, \; {\rm dist} (x, \partial \Omega) \geq \frac{1}{m} \,, \\
    & \phi_m \rightarrow 1 \,, \;\; | \nabla \phi_m | \leq 2 m \, \; \textrm{ for } x \in \Omega \,.
  \end{aligned}
\end{equation*}

Now, we tackle with the term $\mathcal{R}_{\chi \phi_m}$. For the term $\mathscr{X}_1$ in $\mathcal{R}_{\chi \phi_m}$, in view of \eqref{weak-est}, and using the Lebesgue's dominated convergence theorem, one has, as $m \rightarrow \infty$,
\begin{equation}\label{X1}
  \begin{aligned}
    \mathscr{X}_1 \rightarrow 0 \,.
  \end{aligned}
\end{equation}

For the term $\mathscr{X}_2$, by H\"{o}lder's inequality and \eqref{weak-est}, we get, as $m \rightarrow \infty$,
\begin{align}\label{X2}
    \mathscr{X}_2 & \leq \int_0^T \chi \int ( \rho | {\bf u} |^2 + {\rm p}(\rho) ) | \nabla \phi_m {\rm dist} (x, \partial \Omega) | | {\bf v} [ {\rm dist} (x, \partial \Omega) ]^{-1} | \dd x \dd t \nonumber\\
    & \leq C \int_0^T \| {\bf v} [ {\rm dist} (x, \partial \Omega) ]^{-1} \|_{L^\infty} \int_{ \{ x : \, {\rm dist} (x, \partial \Omega) \leq \frac{1}{m} \} } \rho | {\bf u} |^2 + {\rm p}(\rho) \dd x \dd t \nonumber\\
    & \leq C \sup_{0 \leq t \leq T} \int_{ \{ x : \, {\rm dist} (x, \partial \Omega) \leq \frac{1}{m} \} } \rho | {\bf u} |^2 + {\rm p}(\rho) \dd x \times \| {\bf v} [ {\rm dist} (x, \partial \Omega) ]^{-1} \|_{ L^1 ( 0, T; L^\infty ) } \nonumber\\
    & \rightarrow 0 \,.
\end{align}
Here, we have employed the fact that ${\bf v} [ {\rm dist} (x, \partial \Omega) ]^{-1} \in L^1(0, T; L^\infty)$. Indeed, by Hardy's inequality,
\begin{equation*}
  \begin{aligned}
    \| {\bf v} [ {\rm dist} (x, \partial \Omega) ]^{-1} \|_{L^p} \leq C \| \nabla {\bf v} \|_{L^p} \leq C \| \nabla {\bf v} \|_{L^\infty} \,, \text{ for } 1 < p < \infty \,,
  \end{aligned}
\end{equation*}
together with ${\bf v} \in L^1 (0, T; W^{1,\infty})$ (see Remark \ref{grad-v-regularity}) and Lemma \ref{Lp-limit-L-infty}, the conclusion holds.

With a similar argument to the term $\mathscr{X}_2$, by H\"{o}lder's inequality and Hardy's inequality, along with the regularity of $r$ and ${\bf v}$ in Theorem \ref{thm-weak-dissipative}, it infers that, as $m \rightarrow \infty$,
\begin{align}
    \mathscr{X}_3 & \leq \int_0^T \chi \int | \mathbb{S} ( \nabla {\bf u} ) | | \nabla \phi_m \rm{dist} (x, \partial \Omega) |  | {\bf v}[ {\rm dist} (x, \partial \Omega) ]^{-1} | \dd x \dd t \nonumber\\
    & \leq C \| \nabla {\bf u} \|_{ L^2(0, T; L^2) } \Big ( \int_0^T\int_{ \{ x : \, {\rm dist} (x, \partial \Omega) \leq \frac{1}{m} \} } | \nabla {\bf v} |^2 \dd x \dd t \Big )^{\frac12} \rightarrow 0 \,, \label{X3}\\
    \mathscr{X}_4 & \leq \int_0^T \chi \int \rho | {\bf u} | | \nabla \phi_m {\rm dist} (x, \partial \Omega) | \frac12 | {\bf v} |^2 [ {\rm dist} (x, \partial \Omega) ]^{-1} \dd x \dd t \nonumber\\
    & \leq C \| \rho {\bf u} \|_{ L^\infty(0, T; L^{\frac{2\gamma}{\gamma+1}}) } \int_0^T \Big ( \int_{ \{ x : \, {\rm dist} (x, \partial \Omega) \leq \frac{1}{m} \} } | {\bf v} \cdot \nabla {\bf v} |^{\frac{2\gamma}{\gamma - 1}} \dd x \Big )^{\frac{\gamma - 1}{2\gamma}} \dd t \rightarrow 0 \,, \label{X4}\\
    \mathscr{X}_5 & \leq \int_0^T \chi \int \rho | {\bf u} | | \nabla \phi_m {\rm dist} (x, \partial \Omega) | {\rm P}'(r) [ {\rm dist} (x, \partial \Omega) ]^{-1} \dd x \dd t \nonumber\\
    & \leq C \| \rho {\bf u} \|_{ L^\infty(0, T; L^{\frac{2\gamma}{\gamma+1}}) } \int_0^T \Big ( \int_{ \{ x : \, {\rm dist} (x, \partial \Omega) \leq \frac{1}{m} \} } | \nabla r |^{\frac{2\gamma}{\gamma - 1}} \dd x \Big )^{\frac{\gamma - 1}{2\gamma}} \dd t \rightarrow 0 \,, \label{X5}\\
    \mathscr{X}_6 & \leq \int_0^T \chi \int | {\rm p}(r) | | \nabla \phi_m {\rm dist} (x, \partial \Omega) | | {\bf v} [ {\rm dist} (x, \partial \Omega) ]^{-1} | \dd x \dd t \nonumber\\
    & \leq C \sup_{0 \leq t \leq T} \int_{ \{ x : \, {\rm dist} (x, \partial \Omega) \leq \frac{1}{m} \} } {\rm p}(r) \dd x \| {\bf v} [ {\rm dist} (x, \partial \Omega) ]^{-1} \|_{ L^1 ( 0, T; L^\infty ) } \rightarrow 0 \,. \label{X6}
\end{align}
Then, we have, as $m \rightarrow \infty$,
\begin{equation*}
  \begin{aligned}
    \mathcal{R}_{\chi \phi_m} \rightarrow 0 \,.
  \end{aligned}
\end{equation*}

Next, letting $m \rightarrow \infty$, if follows from \eqref{ws-relative-energy1} that
\begin{align*}
    & - \int_0^T \partial_t \chi \int \Big [ - \rho {\bf u} \cdot {\bf v} + \frac{1}{2} | {\bf v} |^2 - {\rm P}'(r) \rho + {\rm P}'(r) r - {\rm P}(r) \Big ] \dd x \dd t \nonumber\\
    = & - \int_0^T \chi \int \rho ( {\bf u} - {\bf v} ) \cdot \nabla {\bf v} \cdot ( {\bf u} - {\bf v} ) \dd x \dd t \nonumber\\
    & + \int_0^T \chi \int \rho ( \partial_t {\bf v} + {\bf v} \cdot \nabla {\bf v} ) \cdot ( {\bf v} - {\bf u} ) \dd x \dd t \nonumber\\
    & + \int_0^T \chi \int ( r - \rho ) \partial_t {\rm P}'(r) \dd x \dd t + \int_0^T \chi \int ( r {\bf v} - \rho {\bf u} ) \cdot \nabla {\rm P}'(r) \dd x \dd t \nonumber\\
    & - \int_0^T \chi \int ( {\rm p}(\rho) - {\rm p}(r) ) \dv {\bf v} \dd x \dd t + \int_0^T \chi \int \mathbb{S} (\nabla {\bf u}) : \nabla {\bf v} \dd x \dd t \,.
\end{align*}

Since $\chi \in C_c^\infty ( (0, T) )$ is arbitrary, then it holds that, for a.e. $t \in (0, T)$,
\begin{align}\label{ws-relative-energy-part}
    & {\rm \frac{d}{dt}} \int \Big [ - \rho {\bf u} \cdot {\bf v} + \frac{1}{2} | {\bf v} |^2 - {\rm P}'(r) \rho + {\rm P}'(r) r - {\rm P}(r) \Big ] \dd x \nonumber\\
    = & - \int \rho ( {\bf u} - {\bf v} ) \cdot \nabla {\bf v} \cdot ( {\bf u} - {\bf v} ) \dd x + \int \rho ( \partial_t {\bf v} + {\bf v} \cdot \nabla {\bf v} ) \cdot ( {\bf v} - {\bf u} ) \dd x \nonumber\\
    & + \int ( r - \rho ) \partial_t {\rm P}'(r) \dd x \dd t + \int ( r {\bf v} - \rho {\bf u} ) \cdot \nabla {\rm P}'(r) \dd x \nonumber\\
    & - \int ( {\rm p}(\rho) - {\rm p}(r) ) \dv {\bf v} \dd x + \int \mathbb{S} (\nabla {\bf u}) : \nabla {\bf v} \dd x \,.
\end{align}

Adding ${\rm \frac{d}{dt}} \int \frac{1}{2} \rho | {\bf u} |^2 + {\rm P}(\rho) \dd x + \int \mathbb{S} (\nabla {\bf u} ) : \nabla {\bf u} \dd x$ to both sides of \eqref{ws-relative-energy-part}, by means of the definitions $E_1 (r, {\bf v})$ and $E_2 (r, {\bf v})$ in \eqref{E1-2}, we have, for a.e. $t \in (0, T)$,
\begin{align}\label{ws-relative-energy-D}
    & {\rm \frac{d}{dt}} \mathcal{E}_{ws} ( \rho, {\bf u}; r, {\bf v} ) + \int \mathbb{S} (\nabla {\bf u} - \nabla {\bf v} ) : \nabla ( {\bf u} - {\bf v} ) \dd x \nonumber\\
    = & {\rm \frac{d}{dt}} \int \frac{1}{2} \rho | {\bf u} |^2 + {\rm P}(\rho) \dd x + \int \mathbb{S} (\nabla {\bf u} ) : \nabla {\bf u} \dd x \nonumber\\
    & - \int \rho ( {\bf u} - {\bf v} ) \cdot \mathbb{D} ( {\bf v} ) \cdot ( {\bf u} - {\bf v} ) \dd x - \int ( {\rm p}(\rho) - {\rm p}'(r) (\rho - r) - {\rm p}(r) ) \dv {\bf v} \dd x \nonumber\\
    & + \int ( r - \rho ) {\rm P}''(r) E_1 (r, {\bf v}) \dd x + \int \frac{\rho}{r} E_2 (r, {\bf v}) \cdot ( {\bf v} - {\bf u} ) \dd x \nonumber\\
    & + \int \frac{\rho - r}{r} \dv \mathbb{S} (\nabla {\bf v}) \cdot ( {\bf v} - {\bf u} ) \dd x \,,
\end{align}
where
\begin{equation*}
  \begin{aligned}
    \mathcal{E}_{ws} ( \rho, {\bf u}; r, {\bf v} ) = \int \frac{1}{2} \rho | {\bf u} - {\bf v} |^2 + {\rm P}(\rho) - {\rm P}'(r)( \rho - r ) - {\rm P}(r) \dd x \,.
  \end{aligned}
\end{equation*}

With the same way as \eqref{r-S-u-v-torus} to deal with the last term on the right-hand side of \eqref{ws-relative-energy-D}, applying Gr\"{o}nwall's inequality, and using the energy inequality \eqref{ws-weak-energy-ineq} in Definition \ref{def-weak-sol}, it yields, for $0 < \tilde{s} < t < T$,
\begin{align*}
    & \mathcal{E}_{ws} ( \rho, {\bf u}; r, {\bf v} ) ( t ) + \frac{1}{2} \int_{ \tilde{s} }^t \int \mathbb{S} (\nabla {\bf u} - \nabla {\bf v} ) : \nabla ( {\bf u} - {\bf v} ) \dd x \dd s \\
    \leq & \exp \Big ( \int_{ \tilde{s} }^t C_0 \Lambda ( {\bf v} ) \dd s \Big ) \mathcal{E}_{ws} ( \rho, {\bf u}; r, {\bf v} ) ( \tilde{s} ) \\
    & + \int_{ \tilde{s} }^t \int \exp \Big ( \int_s^t C_0 \Lambda ( {\bf v} ) \dd \tau \Big ) \big | ( r - \rho ) {\rm P}''(r) E_1 (r, {\bf v}) \big | \dd x \dd s \\
    & + \int_{ \tilde{s} }^t \int \exp \Big ( \int_s^t C_0 \Lambda ( {\bf v} ) \dd \tau \Big ) \Big | \frac{\rho}{r} E_2 (r, {\bf v}) \cdot ( {\bf v} - {\bf u} ) \Big | \dd x \dd s \,.
\end{align*}

Using \eqref{weak-conti-D}, \eqref{weak-mome-D} and \eqref{weak-est} in Definition \ref{def-weak-sol}, and making the density arguments, the weak solution $(\rho, {\bf u})$ belongs to the regularity class
\begin{equation}\label{ws-conti-t}
  \begin{aligned}
    \rho \in C_w ( [0, T]; L^\gamma ) \,, \quad \rho {\bf u} \in C_w ( [0, T]; L^{ \frac{2\gamma}{\gamma+1} } ) \,.
  \end{aligned}
\end{equation}

Choosing $\tilde{s} : = s_n$ with $s_n \rightarrow 0 \; (n \rightarrow \infty)$, and taking advantage of \eqref{ws-weak-energy-ineq} and \eqref{ws-conti-t}, it implies that
\begin{align*}
    \mathcal{E}_{ws} ( \rho, {\bf u}; r, {\bf v} ) ( s_n ) \leq & \int \frac{1}{2} \frac{| {\bf m}_0 |^2}{\rho_0} + {\rm P}(\rho_0) \dd x - \int \rho {\bf u} \cdot {\bf v} \dd x \Big |_{t = s_n} + \int \frac12 \rho | {\bf v} |^2 \dd x \Big |_{t = s_n} \\
    & - \int \rho {\rm P}'(r) \dd x \Big |_{t = s_n} + \int {\rm P}'(r) r - {\rm P}(r) \dd x \Big |_{t = s_n} \\
    \rightarrow & \int \frac{1}{2} \rho_0 \Big | \frac{{\bf m}_0}{\rho_0} - {\bf v}_0 \Big |^2 + {\rm P}( \rho_0 ) - {\rm P}'( r_0 )( \rho_0 - r_0 ) - {\rm P}( r_0 ) \dd x \\
    = : & \mathcal{E}_{ws} ( \rho_0, {\bf m}_0; r_0, {\bf v}_0 ) \,, \textrm{ as } n \rightarrow \infty \,.
\end{align*}

Thus, we complete the proof of Theorem \ref{thm-weak-dissipative}.

\appendix
\section{}\label{Sec:appendix}

In this appendix, we first collect the classical inequality, some properties of $L^p$ space, and two weak convergence results. Then, we add two supplementary lemmas (Lemmas \ref{Korn-type-ineq} and \ref{mollifier-L-infty}) and give the proofs of them. These facts are frequently used in the proof of our main results. We point out that the Lemma \ref{mollifier-L-infty} subjects to the density arguments to deduce \eqref{v-delta}, \eqref{r-delta-infty} and \eqref{p'(r)-div-v-delta} in Section \ref{Sec:regular}.

\begin{lem}[\cite{Brezis-book-2010}, Chapter 9, Poincar\'{e}-Wirtinger's inequality]\label{Poincare-ineq}
Let $\Omega \subset \mathbb{R}^n$ be a bounded domain with a $C^1$ boundary $\partial \Omega$. Then,
\begin{equation*}
  \begin{aligned}
    \Big \| w - \frac{1}{|\Omega|} \int_{\Omega} w \dd x \Big \|_{L^q (\Omega)} \leq C \| \nabla w \|_{L^p (\Omega)} \,, \;\; \forall w \in W^{1,p} \,,
  \end{aligned}
\end{equation*}
where $1 \leq q \leq \frac{pn}{n-p}$ for $n > p$.
\end{lem}

\begin{lem}\label{Korn-type-ineq}
Let $\Omega \subset \mathbb{R}^{n}$ be a bounded domain with a $C^1$ boundary $\partial \Omega$, and a non-negative function $\tilde{r}$ satisfies
\begin{equation*}
  \begin{aligned}
    0 < M_0 \leq \int_{\Omega} \tilde{r} \dd x \,, \quad \int_{\Omega} \tilde{r}^\gamma \dd x \leq M_1 \,,
  \end{aligned}
\end{equation*}
for some positive constants $M_0$ and $M_1$, where it assumes $\gamma \geq \frac{2n}{n+2}$ for $n \geq 3$. Then, there exists a constant $\tilde{c}_{\gamma} : = C(n, \gamma, \Omega, M_0, M_1)$ such that
\begin{equation*}
  \begin{aligned}
    \| w \|_{H^1 (\Omega)} \leq \tilde{c}_{\gamma} \big ( \| \nabla w \|_{L^2 (\Omega)} + \| \sqrt{\tilde{r}} w \|_{L^2 (\Omega)} \big ) \,.
  \end{aligned}
\end{equation*}
\end{lem}

\begin{rem}
Lemmas \ref{Poincare-ineq} and \ref{Korn-type-ineq} still hold when $\Omega = \mathbb{T}^n$. Lemma \ref{Korn-type-ineq} can be directly deduced by Theorem 11.23 called generalized Korn-Poincar\'{e} inequality in \cite{FN-book-2017}. For reader's convenience, we give a brief proof by means of Lemma \ref{Poincare-ineq} as follows.
\end{rem}

\begin{proof}
By Minkowski's inequality and Poincar\'{e}'s inequality in Lemma \ref{Poincare-ineq}, it follows that
\begin{equation*}
  \begin{aligned}
    \| w \|_{L^2 (\Omega)} \leq \Big \| w - \frac{1}{|\Omega|} \int_{\Omega} w \dd x \Big \|_{L^2 (\Omega)} + \Big \| \frac{1}{|\Omega|} \int_{\Omega} w \dd x \Big \|_{L^2 (\Omega)} \leq C \| \nabla w \|_{L^2 (\Omega)} + |\Omega|^{- \frac12} \| w \|_{L^1 (\Omega)} \,.
  \end{aligned}
\end{equation*}
Then, we know
\begin{equation*}
  \begin{aligned}
    \| w \|_{H^1 (\Omega)} \leq C \| \nabla w \|_{L^2 (\Omega)} + C \int_{\Omega} | w | \dd x \,.
  \end{aligned}
\end{equation*}
Using H\"{o}lder's inequality, one has
\begin{equation*}
  \begin{aligned}
    \int_{\Omega} r \dd x \frac{1}{ |\Omega| } \int_{\Omega} | w | \dd x \leq & \int_{\Omega} r \Big | w - \frac{1}{ |\Omega| } \int w \dd x \Big | \dd x + \int_{\Omega} r | w | \dd x \\
    \leq & \| r \|_{ L^{ \frac{p}{p-1} } (\Omega) } \Big \| w - \frac{1}{ |\Omega| } \int_{\Omega} w \dd x \Big \|_{L^p (\Omega)} + \| \sqrt{r} \|_{L^2 (\Omega)} \| \sqrt{r} w \|_{L^2 (\Omega)} \,.
  \end{aligned}
\end{equation*}
Taking $p = \frac{2n}{n - 2}$, it sees that $\frac{p}{p - 1} = \frac{2n}{n + 2}$. By the Poincar\'{e}-Wirtinger's inequality stated in Lemma \ref{Poincare-ineq} and the restrictions on $\gamma$, the conclusion follows.
\end{proof}

\begin{lem}[\cite{ChenShX-book-2018}, Lemma 3.1 in Chapter 4]\label{Lp-limit-L-infty}
Let $\Omega \subset \mathbb{R}^n$ be a bounded domain. If $w \in L^p (\Omega)$ for $1 \leq p < \infty$, and $\| w \|_{ L^p (\Omega) } \leq M < \infty$, where the constant $M$ is independent of $p$. Then, $w \in L^\infty (\Omega)$, and as $p \rightarrow \infty$,
\begin{equation*}
  \begin{aligned}
    \| w \|_{L^p (\Omega)} \rightarrow \| w \|_{L^\infty (\Omega)} \,.
  \end{aligned}
\end{equation*}

\end{lem}

\begin{lem}[\cite{Brezis-book-2010}, Lemma 4.3 in Chapter 4]\label{local-unif-conti-Lp}
Let $\Omega \subset \mathbb{R}^n$ be an open set, $K \subset\subset \Omega$, $K = \Omega$ if $\Omega = \mathbb{T}^n$ or $\mathbb{R}^n$. Then, for $w \in L^p (\Omega)$, $1 \leq p < \infty$, as $\xi \rightarrow 0$,
\begin{equation*}
  \begin{aligned}
    \| w (x+\xi) - w (x) \|_{L^p ( K )} \rightarrow 0 \,.
  \end{aligned}
\end{equation*}
\end{lem}

\begin{lem}[\cite{Brezis-book-2010}, Theorem 4.15 in Chapter 4]\label{convolution}
Let $f \in L^1 ( \mathbb{R}^n )$ and $g \in L^p ( \mathbb{R}^n )$ with $1 \leq p \leq \infty$. Define
\begin{align*}
    ( f \ast g ) (x) = \int_{\mathbb{R}^n} f (x-y) g (y) \dd y \,.
\end{align*}
Then,
\begin{equation*}
  \begin{aligned}
    \| f \ast g \|_{L^p ( \mathbb{R}^n )} \leq \| f \|_{L^1 (\mathbb{R}^n)} \| g \|_{L^p (\mathbb{R}^n)} \,.
  \end{aligned}
\end{equation*}
\end{lem}

The following three lemmas are about mollifier. Let $w$ be a locally integrable function. Recall the definitions of mollifiers $\eta (x)$ in \eqref{mollifier1} and $\tilde{\eta} (t)$ in \eqref{mollifier2}. We use $w_x^\delta$ to denote the mollification of $w$ with respect to $x$, and $w_{t,x}^\delta$ to denote the mollification of $w$ with respect to both $x$ and $t$, that is,
\begin{equation*}
  \begin{aligned}
    & w_x^\delta (t, x) = ( w \ast \eta^\delta ) (t, x) = \int_{ \mathbb{R}^n } w (t, x - y) \eta^\delta (y) \dd y \,, \\
    & w_{t,x}^\delta (t, x) = ( w^\delta \ast \tilde{\eta}^\delta ) (t, x) = \int_{ \mathbb{R} } \int_{ \mathbb{R}^n } w (s, x - y) \eta^\delta (y) \dd y \tilde{\eta}^\delta (t-s) \dd s \,.
  \end{aligned}
\end{equation*}

\begin{lem}\label{mollifier-derivative}
Let $w \in W^{k,p}$ with $1 \leq p < \infty$, then, for $|\alpha| \leq k$,
\begin{align*}
    D^\alpha w_x^\delta = (D^\alpha w)_x^\delta \,.
\end{align*}
\end{lem}

\begin{rem}
This lemma is one conclusion during the proof of Theorem 1 in Section 5.3 of \cite{Evans-book-2010}, which means that the $\alpha^{th}$ order partial derivative of the smooth function $w_x^\delta$ is the mollification of the $\alpha^{th}$ order weak partial derivative of $w$. Similarly, the regularizing in $t$ has the same property. We directly write it as follows, if $\partial_t w \in L^q (0, T; L^p)$, $1 \leq p, q < \infty$, then $\partial_t w_{t,x}^\delta = (\partial_t w)_{t,x}^\delta$.
\end{rem}

\begin{lem}[\cite{Brezis-book-2010}, Theorem 4.22 in Chapter 4]\label{mollifier-Lp}
Let $\Omega \subset \mathbb{R}^n$ be an open set, $K \subset\subset \Omega$, $K = \Omega$ if $\Omega = \mathbb{T}^n$ or $\mathbb{R}^n$. Then, for $w \in L^p (\Omega)$, $1 \leq p < \infty$,
\begin{equation*}
  \begin{aligned}
    \| w_x^\delta - w \|_{L^p(K)} \rightarrow 0 \, \textrm{ as } \delta \rightarrow 0 \,.
  \end{aligned}
\end{equation*}
\end{lem}

\begin{lem}\label{mollifier-L-infty}
Let $\Omega \subset \mathbb{R}^n$ be a bounded domain, $K \subset\subset \Omega$, $K = \Omega$ if $\Omega = \mathbb{T}^n$, and $(t_1, t_2) \subset\subset (0, T)$ with any $T \in (0, \infty)$. Then \\
1. For $\partial_t w \in L^1 (0, T; L^1 (\Omega))$, $\nabla w \in L^\infty(0, T; L^1 (\Omega))$, we have
\begin{equation*}
  \begin{aligned}
    \| w_{t,x}^\delta - w \|_{ L^\infty (t_1, t_2; L^\infty (K)) } \rightarrow 0 \, \textrm{ as } \delta \rightarrow 0 \,.
  \end{aligned}
\end{equation*}
2. For $\partial_t w \in L^1 (0, T; L^1 (\Omega))$, $w \in L^\infty(0, T; L^p (\Omega))$ with $1 \leq p < \infty$, we have
\begin{equation*}
  \begin{aligned}
    \| w_{t,x}^\delta - w \|_{ L^\infty (t_1, t_2; L^p (K)) } \rightarrow 0 \, \textrm{ as } \delta \rightarrow 0 \,.
  \end{aligned}
\end{equation*}
\end{lem}

\begin{proof}
1. By the definition of mollifiers in \eqref{mollifier1} and \eqref{mollifier2}, and with some direct computations, it has
\begin{align}\label{w-delta}
    & w_{t,x}^\delta (t, x) - w (t, x) = \int_{ \mathbb{R} } \int_{ \mathbb{R}^n } [ w (t - s, x - y) - w (t, x) ] \eta^\delta (y) \dd y \tilde{\eta}^\delta (s) \dd s \nonumber\\
    = & \int_{ \mathbb{R} } \int_{ \mathbb{R}^n } [ w (t - s, x - y) - w (t, x - y) ] \eta^\delta (y) \dd y \tilde{\eta}^\delta (s) \dd s \nonumber\\
    & + \int_{ \mathbb{R}^n }  [ w (t, x - y) - w (t, x) ] \eta^\delta (y) \dd y \nonumber\\
    = & \underbrace{ \int_{ \mathbb{R} } \int_{ \mathbb{R}^n } [ w (t - \delta s, x - \delta y) - w (t, x - \delta y) ] \eta (y) \dd y \tilde{\eta} (s) \dd s }_{ y := y' = \delta y \,, \;\; s := s' = \delta s } \nonumber\\
    & + \underbrace{ \int_{ \mathbb{R}^n }  [ w (t, x - \delta y) - w (t, x) ]  \eta (y) \dd y }_{y : = y' = \delta y} \nonumber\\
    = & \int_{ \mathbb{R} } \int_{ \mathbb{R}^n } \int_0^1 \frac{ {\rm d} }{ {\rm d} \tau } w (t - \tau \delta s, x - \delta y) \dd \tau \eta (y) \dd y \tilde{\eta} (s) \dd s \nonumber\\
    & + \int_{ \mathbb{R}^n } \int_0^1 \frac{ {\rm d} }{ {\rm d} \tau } w (t, x - \tau \delta y) \dd \tau \eta (y) \dd y \nonumber\\
    = & \underbrace{ \int_0^1 \int_{ \mathbb{R} } \int_{ \mathbb{R}^n } \partial_t w (t - \tau \delta s, x - \delta y) \times (- \delta s) \eta (y) \dd y \tilde{\eta} (s) \dd s \dd \tau }_{ \textrm{ Fubini's theorem } } \nonumber\\
    & + \underbrace{ \int_0^1 \int_{ \mathbb{R}^n } \nabla w (t, x- \tau \delta y) \cdot (- \delta y) \eta (y) \dd y \dd \tau }_{ \textrm{ Fubini's theorem } } \nonumber\\
    = : & I + II \,.
\end{align}
For the first part $I$, noticing that ${\rm supp} \ \eta (y) = \{ y \in \mathbb{R}^n : \, |y| \leq 1 \}$ and ${\rm supp} \ \tilde{\eta} (s) = \{ s \in \mathbb{R} : \, |s| \leq 1 \}$, and by H\"{o}lder's inequality, we have
\begin{align}\label{I}
    I \leq & C \delta \int_0^1 \int_{ \{ s : \, |s| \leq 1 \} } \int_{ \{ y : \, |y| \leq 1 \} } | \partial_t w (t - \tau \delta s, x - \delta y) | \dd y \dd s \dd \tau \nonumber\\
    = & C \delta \int_0^1 \underbrace{ \int_{ \{ s : \, |s| \leq \tau \delta \} } \int_{ \{ y : \, |y| \leq \delta \} } | \partial_t w (t - s, x - y) | \dd y \dd s  }_{ s := s' = \tau \delta s \,, \;\; y := y' = \delta y } \dd \tau \nonumber\\
    \leq & C \delta \int_0^1 \int_{ [0, T] \cup ( [0, T] - \{ s : \, |s| \leq \delta \} ) } \int_{ \Omega \cup ( \Omega - \{ y : \, |y| \leq \delta \} ) } | \partial_t w (t, x) | \dd x \dd t \dd \tau \nonumber\\
    \leq & C \delta \| \partial_t w \|_{L^1 (0, T; L^1 (\Omega))} \,,
\end{align}
where $\Omega - \{ y : \, |y| \leq \delta \} = \{ x - y : \, x \in \Omega, y \in \mathbb{R}^n, |y| \leq \delta \}$ and $[0, T] - \{ s : \, |s| \leq \delta \} = \{ t - s : \, t \in [0, T], s \in \mathbb{R}, |s| \leq 1 \}$.

Similarly, for the second part $II$, it holds
\begin{align*}
    II \leq & C \delta \int_0^1 \int_{ \{ y : \, |y| \leq 1 \} } | \nabla w (t, x- \tau \delta y) | \dd y \dd \tau \\
    \leq & C \delta \int_0^1 \int_{ \Omega \cup ( \Omega - \{ y : \, |y| \leq \delta \} } | \nabla w (t, x) | \dd x \dd \tau \\
    \leq & C \delta \| \nabla w (t, \cdot) \|_{ L^1 ( \Omega ) } \,.
\end{align*}

Therefore,
\begin{equation*}
  \begin{aligned}
    \| w_{t,x}^\delta - w \|_{L^\infty ( t_1, t_2; L^\infty (K) ) } \leq C \delta \| \partial_t w \|_{L^1 (0, T; L^1 (\Omega))} + C \delta \| \nabla w \|_{L^\infty ( 0, T; L^1(\Omega) )} \rightarrow 0
  \end{aligned}
\end{equation*}
as $\delta \rightarrow 0$,

2. With the similar arguments to \eqref{w-delta} and \eqref{I}, we have
\begin{align*}
    & w_{t,x}^\delta (t, x) - w (t, x) = \int_{ \mathbb{R} } \int_{ \mathbb{R}^n } [ w (t - s, x - y) - w (t, x) ] \eta^\delta (y) \dd y \tilde{\eta}^\delta (s) \dd s \\
    = & \int_0^1 \int_{ \mathbb{R} } \int_{ \mathbb{R}^n } \partial_t w (t - \tau \delta s, x - \delta y) \times (- \delta s) \eta (y) \dd y \tilde{\eta} (s) \dd s \dd \tau \\
    & + \int_{ \mathbb{R}^n }  [ w (t, x - \delta y) - w (t, x) ] \eta (y) \dd y \\
    \leq & C \delta \| \partial_t w \|_{ L^1 ( 0, T; L^1 (\Omega) )} + C \int_{ \{ y : |y| \leq 1 \} }  | w (t, x - \delta y) - w (t, x) | \dd y \,.
\end{align*}
Using Minkowski's integral inequality, then we arrive at
\begin{align*}
    & \| w_{t,x}^\delta (t, x) - w (t, x) \|_{ L^\infty ( t_1, t_2; L^p (K) ) } \\
    \leq & C \delta \| \partial_t w \|_{ L^1 (0,T; L^1 (\Omega))} + C \sup_{ 0 \leq t \leq T } \int_{ \{ y : |y| \leq 1 \} } \| w (t, x - \delta y) - w (t, x) \|_{ L^p (K) } \dd y \,.
\end{align*}
By Lemma \ref{local-unif-conti-Lp}, it gives
\begin{align*}
    \| w_{t,x}^\delta - w \|_{ L^\infty (t_1, t_2; L^p (K)) } \rightarrow 0 \, \textrm{ as } \delta \rightarrow 0 \,.
\end{align*}
\end{proof}

Finally, we recall two important lemmas to deal with the product of two weak convergence sequences.

\begin{lem}[\cite{Lions-1996}, Lemma C.1]\label{conti-time-weak}
Let $X$ be a reflexive Banach space, $Y$ be a Banach space, $X \hookrightarrow Y$, $Y'$ is separable and dense in $X'$. Assume a sequence $\{ f_n \}$ satisfies $f_n \in L^\infty (0,T; X)$ and $\partial_t f_n \in L^p (0,T; Y)$ with $1 < p \leq \infty$. Then, $f_n$ is relatively compact in $C_w ( [0, T]; X )$.
\end{lem}

\begin{lem}[\cite{Lions-1998}, Lemma 5.1]\label{product-weak-converge}
Assume $g_n \rightharpoonup g$ weakly in $L^{p_1} (0,T; L^{p_2})$, $h_n \rightharpoonup h$ weakly in $L^{q_1} (0,T; L^{q_2})$, $\frac{1}{p_1} + \frac{1}{p_2} = \frac{1}{q_1} + \frac{1}{q_2} = 1$, $1 \leq p_1, p_2 \leq \infty$, $q_1 > 1$. In addition, $\partial_t g_n$ is uniformly bounded in $L^1(0,T; W^{-k,1})$ for some $k \geq 0$, and $\| h_n (x,t) - h_n (x + \xi, t) \|_{ L^{q_1} (0,T; L^{q_2}) } \rightarrow 0$ as $| \xi | \rightarrow 0$ for any $n$. Then, $g_n h_n \rightarrow g h$ in $\mathscr{D}' ( (0, T) \times \Omega )$.
\end{lem}

\medskip \medskip
 \noindent
{\bf Acknowledgements:}  \ \
Liang Guo is   supported  by NSFC (Grant No. 11671193). Fucai Li is  supported in part by NSFC (Grant Nos.11971234 and 11671193)  and a project funded by the priority academic program development of Jiangsu higher education institutions.
Cheng Yu is partially supported by Collaboration Grants for Mathematicians from Simons Foundation.

\bibliographystyle{plain}

\end{document}